\documentclass[12pt]{amsart}
\usepackage{amsmath,amssymb,latexsym,float,epsfig,soul,bm}
\usepackage{wasysym,dynkin-diagrams,mathdots,multirow,microtype,hyperref}
\usepackage[top=1in,bottom=1in,left=1in,right=1in]{geometry}
\usepackage[backend=bibtex, style=alphabetic, maxcitenames=2, maxbibnames=7, uniquename=false, uniquelist=false, citetracker=true, sorting=anyt]{biblatex}
\bibliography{library.bib}
\usepackage{array,tikz-cd}
\usepackage[capitalise]{cleveref}

\def\a{\alpha}
\def\b{\beta}
\def\g{\gamma}
\def\d{\delta}
\def\ep{\varepsilon}

\def\k{\kappa}
\def\l{\lambda}

\def\s{\sigma}
\def\vp{\varphi}

\def\w{\omega}
\def\t{\tau}

\def\R{\mathbb R}

\def\Z{\mathbb Z}

\def\F{\mathbb F}
\def\E{\mathbb E}

\def\f{\mathcal F}
\def\cl{\mathcal L}
\def\ts{\mathcal T}

\renewcommand{\hom}{\text{Hom}}
\def\overbigdot#1{\overset{\hbox{\tiny$\bullet$}}{#1}}
\def\kapb{\overbigdot{\kappa}}
\def\kapc{\mathring{\kappa}}
\def\id{\text{id}}
\def\im{\text{Im}}

\def\G{\overline{G}}
\def\T{\overline{T}}
\def\A{\overline{A}}
\DeclareMathOperator{\syl}{Syl}
\DeclareMathOperator{\out}{Out}
\DeclareMathOperator{\iso}{Iso}
\DeclareMathOperator{\innd}{Inndiag}
\DeclareMathOperator{\od}{Outdiag}
\DeclareMathOperator{\mor}{Mor}
\DeclareMathOperator{\aut}{Aut}
\DeclareMathOperator{\inn}{Inn}
\DeclareMathOperator{\inj}{Inj}
\DeclareMathOperator{\en}{End}
\DeclareMathOperator{\sym}{Sym}
\DeclareMathOperator{\alter}{Alt}
\DeclareMathOperator{\supp}{Supp}
\DeclareMathOperator{\height}{ht}
\DeclareMathOperator{\pro}{pr}

\DeclareMathOperator{\ad}{ad}
\DeclareMathOperator{\un}{un}
\DeclareMathOperator{\PSL}{PSL}
\DeclareMathOperator{\POmega}{P\Omega}
\DeclareMathOperator{\PSU}{PSU}
\DeclareMathOperator{\SL}{SL}

\DeclareMathOperator{\scal}{sc}
\def\ob{\text{Ob}}
\renewcommand{\leq}{\leqslant}
\renewcommand{\hat}{\widehat}
\renewcommand{\geq}{\geqslant}
\renewcommand{\subset}{\subseteq}
\newcommand{\norm}{\trianglelefteq}

\usepackage{amsthm}
\newtheorem{theo}{Theorem}
\newtheorem{lem}[theo]{Lemma}
\newtheorem{prop}[theo]{Proposition}
\newtheorem{con}[theo]{Conjecture}
\newtheorem{cor}[theo]{Corollary}
\theoremstyle{definition}
\newtheorem{defn}[theo]{Definition}

\newtheorem{no}[theo]{Notation}
\newtheorem{hypth}[theo]{Hypotheses}
\newtheorem{ques}[theo]{Question}

\numberwithin{theo}{section}

\makeatletter
\newtheorem{repeattheo@}{Theorem}

\makeatother

\makeatletter
\newtheorem{repeatlem@}{Lemma}

\makeatother

\makeatletter
\newtheorem{repeatcor@}{Corollary}
\newenvironment{repeatcor}[1]{%
	\def\therepeatcor@{\ref{#1}}
	\repeatcor@
}
{\endrepeatcor@}
\makeatother

\renewcommand{\arraystretch}{1.5}
\newcolumntype{C}[1]{>{\centering\let\newline\\\arraybackslash\hspace{0pt}}m{#1}}

\usepackage[
textwidth=3cm,
textsize=small,
colorinlistoftodos]
{todonotes}

\begin{document}	
	\title[Rigid automorphisms associated to finite groups of Lie type]{Rigid automorphisms of linking systems of finite groups of Lie type}
	\author{Jonathon Villareal}
	\begin{abstract}
		Let $\cl$ be a centric linking system associated to a saturated fusion system on a finite $p$-group $S$. An automorphism of $\cl$ is said to be rigid if it restricts to the identity on the fusion system. An inner rigid automorphism is conjugation by some element of the center of $S$. If $\cl$ is the centric linking system of a finite group $G$, then rigid automorphisms of $\cl$ are closely related to automorphisms of $G$ that centralize $S$. For odd primes, all rigid automorphisms are known to be inner, but this fails for the prime 2. We determine which known quasisimple linking systems at the prime 2 have a noninner rigid automorphism. Based on previous results, this reduces to handling the case of the linking systems at the prime 2 of finite simple groups of Lie type in odd characteristic. These have no noninner rigid automorphisms with two families of exceptions: the 2-dimensional projective special linear groups and even-dimensional orthogonal groups for quadratic forms of nonsquare discriminant.
	\end{abstract}
	\maketitle

	%___________________________________________________________
	%___________________________________________________________
	
	%-----------------%
	% 1. INTRODUCTION %
	%-----------------%
	\section{Introduction} 
    Fix a prime $p$ and a finite $p$-group $S$. A saturated fusion system $\f$ over the finite $p$-group $S$ is a category whose objects are subgroups of $S$ and whose morphism sets are given by injective group homomorphisms subject to axioms first given by Puig and Broto, Levi, and Oliver \cite{puig2006frobenius, broto2003homotopy} and later refined by Roberts and Sphectorov \cite{satfussys}. When $G$ is a finite group and $S$ is a Sylow $p$-subgroup of $G$, the category $\f_S(G)$ whose objects are subgroups of $S$ and whose morphisms are injective group homomorphisms between subgroups of $S$ induced by conjugation by elements of $G$ is a saturated fusion system. The importance of $\f_S(G)$ is that it encodes the information required to recover the homotopy type of the Bousfield-Kan $p$-completion $BG^\wedge_p$ of the classifying space of $G$. This was made more precise by the Martino-Priddy conjecture which was proven in \cite{martino1996unstable} and \cite{oliver2004equivalences,oliver2006equivalences}. To recover $BG^\wedge_p$, or more generally $B\f$ when no ambient group is given, one requires the centric linking system $\cl$ associated to $\f$.  \par 

    In \cite{chermak2013fusion} and \cite{oliver2013existence}, Chermak and Oliver show that centric linking systems exist and are unique up to rigid automorphism. Therefore, rigid automorphisms measure the failure of strict uniqueness of centric linking systems. In particular, noninner rigid automorphisms induce nontrivial symmetries of the centric linking system that become trivial upon passing to the associated saturated fusion system. In \cite{oliver2013existence}, Oliver made this precise via an obstruction-theoretic exact sequence that relates $\out(\f)$ and $\out(\cl)$:
    \begin{align*}
        1\to \varprojlim_{\mathcal{O}(\mathcal{F})^c}\nolimits^{1}(\mathcal{Z}_{\f})\to\out(\cl)\xrightarrow{\mu_{\cl}}\out(\f)\to 1
    \end{align*}
    Moreover, the higher limits vanish for $k\geq2$. The group $\ker(\mu_{\cl})$ precisely consists of noninner rigid automorphisms of the centric linking system. \par 
    
    By results of Oliver \cite{oliver2013existence} and Glauberman and Lynd in \cite{glauberman2016control}, it is known that when $p$ is odd, $\ker(\mu_{\cl})$ vanishes, meaning that there are no noninner rigid automorphisms for odd primes. However at the prime 2, there are known examples where this limit is nontrivial (for example, when $G=\alter(6)$). In \cite{andersen2012reduced}, Andersen, Oliver, and Ventura show that when $p=2$, $\ker(\mu_{\cl})$ is a 2-group, and in \cite{glauberman2021rigid}, Glauberman and Lynd refine this result to show that $\ker(\mu_{\cl})$ is an elementary abelian 2-group. 
    
    In this paper, we study the relationship between the automorphisms of a finite simple group of Lie type in odd characteristic and those of the associated fusion and centric linking system at the prime 2. In particular, we extend Broto, M{\o}ller, and Oliver's analysis of the universal groups of Lie type in \cite{broto2019automorphisms}. This gives a full classification of the groups of Lie type whose associated centric linking system have noninner rigid automorphisms. Combining this result with several known results allows us to give a full classification of known quasisimple saturated fusion systems at the prime two which have noninner rigid automorphisms. \par

    For a finite simple group of Lie type $G$ over a field of odd characteristic, we show that the linking system associated to $G$ has a noninner rigid automorphism if and only if the fusion system of $G$ is isomorphic to either the fusion system of a 2-dimensional projective special linear group with nonabelian Sylow 2-subgroups or the fusion system of an orthogonal group of even dimension at least 6 for a quadratic form of nonsquare discriminant. It is known that the universal versions of the finite groups of Lie type over a field in odd characteristic do not have a noninner rigid automorphism at the prime 2 as a result of Oliver \cite[Appendix]{broto2019automorphisms}. We extend the results of \cite{broto2019automorphisms} concerning universal group of Lie type and combine them with a result of \cite{glauberman2021rigid} to show our main result: 
	\begin{theo}[Main Result]\label{theo: main result}
		Let $G$ be a finite quasisimple group of Lie type over a field of odd characteristic. Let $S\in\syl_2(G)$ and set $\f=\f_S(G)$. The associated centric linking system of $\f$ has a noninner rigid automorphism if and only if there is a finite group $G^*$ and $S^*\in\syl_2(G^*)$ such that $\f\cong\f_{S^*}(G^*)$ such that one of the following holds: 
		\begin{enumerate}
			\item $G^*\cong \PSL_2(q)$ with $q\equiv 1\pmod{8}$; or 
			\item $G^*\cong \POmega^{-}_{2n}(q)$ with $n\geq3$ and $q\equiv1\pmod{4}$. 
		\end{enumerate}
		Moreover, in each of these cases, $\ker(\mu_\cl)$ is cyclic of order 2.
	\end{theo}
	
	The classification of known quasisimple saturated fusion systems at the prime 2 whose associated centric linking system have noninner rigid automorphisms is as follows: 
	\begin{cor}\label{cor: quasisimple sat fus with noninner rigid autos}
		Let $\f$ be a known quasisimple saturated fusion system at the prime 2. The associated linking system of $\f$ has a noninner rigid automorphism if and only if there is a finite group $G$ and $S\in\syl_2(G)$ such that $\f\cong\f_S(G)$ and one of the following holds:
		\begin{enumerate}
			\item $G\cong\PSL_2(q)$ with $q\equiv\pm1\pmod{8}$; or
			\item $G\cong\POmega^{-}_{2n}(q)$ with $n\geq3$ and $q\equiv1\pmod{4}$; or
			\item $G\cong\alter(n)$ with $n\equiv2,3\pmod{4}$.
		\end{enumerate}
		Moreover, in each of these cases, $\ker(\mu_\cl)$ is cyclic of order 2.
	\end{cor} 
	
	To prove Theorem \ref{theo: main result}, we first make the same technical assumptions and have similar hypotheses to those of \cite{broto2019automorphisms} (see Hypotheses \ref{hyp: gp of lie type}). We fix a group $G$ that, due to the hypotheses, is a group of Lie type of universal type and is an appropriate representative of its 2-fusion equivalence class. Set $\Gamma_G$ to be the group of all graph automorphisms of $G$ and $\Phi_G$ to be the group of all field automorphisms of $G$. Denote the root system of $G$ by $\Sigma$ and the Weyl group of $G$ by $W_0$.  \par 
	
	Towards determining which finite simple groups of Lie type have noninner rigid automorphisms, we analyze how automorphisms of $G$ act on the quotient of a Sylow 2-subgroup $A$ of a fixed maximal torus $T$ by the center of $\f=\f_S(G)$. We first examine how the automorphisms induced by isometries of $\Sigma$ act on $A/Z(\f)$, and determine that the isometries of $\Sigma$ that are induced by $W_0$ act faithfully on $A/Z(\f)$ under the standing hypotheses. We prove that the field and graph automorphisms for the finite simple groups of Lie type, apart from the two families of exceptions above, act faithfully on $A/Z(\f)$ and then use this to prove an extension of a key lemma of \cite{broto2019automorphisms} (see Lemma \ref{lem: 27}). With this result, we prove Theorem \ref{theo: main result}. A combination of this result with \cite{glauberman2021rigid} proves that there are no rigid automorphisms for the simple groups of Lie type save for the two families of exceptions. Furthermore, constructions of the nonidentity rigid automorphisms for $\PSL_2(q)$ with $q\equiv\pm1\pmod{8}$ and $\POmega^{-}_{2n}(q)$ with $n\geq3$ and $q\equiv1\pmod{4}$are given in Section 5 (see Theorem \ref{theo: A1 has a rigid auto} and Theorem \ref{theo: twisted Dn has a rigid auto}). \par 
	
	The collection of known quasisimple saturated fusion systems at the prime 2 (see Definition \ref{def: known quasisimple saturated fusion systems at the prime 2}) is defined to consist of the fusion systems $\f$ such that $\f$ is a Benson-Solomon fusion system (see \cite{levi2002construction},\cite{aschbacher2010bensonsolomon}) or $\f$ is the fusion system of a known quasisimple finite group $G$ such that $G$ is not 2-Goldschmidt. By a result of Goldschmidt (see \cite{goldschmidt19742} for a classification), a nonabelian finite simple group is $2$-Goldschmidt if and only if it has abelian Sylow 2-groups or is of Lie type in characteristic 2 of Lie rank 1. The $2$-Goldschmidt simple groups have been classified and are 
	\begin{align*}
		&\text{1. }\PSL_2(q)\text{ with }q\equiv\pm3\pmod{8} &  &\text{2. }{}^2G_2(3^{2n+1})\text{ with }n\geq0 \\ 
		&\text{3. }J_1                                      &  &\text{4. }\PSL_2(2^n)\text{ with }n\geq2 \\  
		&\text{5. }\PSU_3(2^n)\text{ with }n\geq2            &  &\text{6. }{}^2B_2(2^{n+1})\text{ with }n\geq1
	\end{align*}
	The known quasisimple groups that are 2-Goldschmidt are the quasisimple coverings of the simple groups listed above. Corollary~\ref{cor: quasisimple sat fus with noninner rigid autos} follows by combining Theorem \ref{theo: main result} with previously known results regarding rigid automorphisms of the Benson-Solomon systems (\cite[Lemma~3.2]{levi2002construction}), sporadic groups (\cite{broto2019automorphisms}), and the alternating groups (\cite{andersen2012reduced}). \par 
	
	Theorem \ref{theo: main result} and Corollary~\ref{cor: quasisimple sat fus with noninner rigid autos} are related to an open question of Oliver (see Question 7.9 in \cite{aschbacher2016fusion}). Let $G$ be a finite 2-perfect group whose Schur multiplier has odd order. Set $\f=\f_S(G)$ and $\cl=\cl_S^c(G)$ for some $S\in\syl_2(G)$. Oliver asked whether the morphism $\mu_\cl$ is injective. Oliver's question can be extended as follows: 
	\begin{con}\label{con: oli con extended}
		If $\f$ is a saturated fusion system and $\cl$ is its associated centric linking system, then $\ker(\mu_\cl)$ is isomorphic to a subgroup of the Schur multiplier of $\f$.
	\end{con}
	
	The Schur multiplier of a fusion system on a $p$-group was defined by Linckelmann in  \cite{linckelmann2006note}; it coincides with the $p$-primary part of the Schur multiplier of $G$ when $\f=\f_S(G)$. As special case of Conjecture \ref{con: oli con extended}, if the Schur multiplier of the fusion system has odd order (as in the case of Oliver's original question), then $\ker(\mu_\cl)=1$ since $\ker(\mu_\cl)$ is a 2-group.
	
	\subsection*{Notation and terminology}
	We fix the following notation and terminology. When $ G $ is a group, $ P\leq G $, and $ g\in G $, we write $ {}^gP=gPg^{-1} $. When $ g\in G $, we write $ c_g $ for the left-handed conjugation homomorphism $ x\mapsto gxg^{-1} $ and its restrictions, and write ${}^gx$ for the image of $c_g(x)$. For any $x,y\in G$, we write $[x,y]=xyx^{-1}y^{-1}={}^xyy^{-1}$. For any subgroups $P,Q\leq G$, we write
	\begin{align*}
		\hom_G(P,Q)&=\{c_g:P\to Q\mid g\in G,\text{ }{}^gP\leq Q\} \\
		\aut_G(P)&=\{c_g:P\to P\mid g\in N_G(P)\} \\
		\inj(P,Q)&=\{\vp\in\hom(P,Q)\mid\vp\text{ is injective}\}
	\end{align*}
	Given a prime $ p $, we denote by $\syl_p(G)$ the set of all Sylow $p$-subgroups of $G$. We also denote the symmetric group and the alternating group on $n$ letters as $\sym(n)$ and $\alter(n)$, respectively. We denote the cyclic group of order $n$ as $Z_n$. Lastly, for any $\a\in\aut(G)$, we denote its restriction to a subgroup $H$ by $\a|_H$. If $N$ is a normal subgroup of $G$ such that $\a(N)=N$, then denote by $\a|_{G/N}$ the induced automorphism of $G/N$.

	%___________________________________________________________	
	
	%-------------------------------------%
	% 2. Background on groups of Lie type %
	%-------------------------------------%
	\section{Background on fusion systems and centric linking systems}
	We primarily follow the notation and definitions set forth in \cite{aschbacher2011fusion}, \cite{aschbacher2016fusion}, and \cite{craven2011theory}. Throughout this section, $p$ denotes a fixed prime. 
	
	\subsection{Basic properties and results about fusion systems}
	\begin{defn}[\cite{aschbacher2011fusion}, \cite{aschbacher2016fusion}]\label{def: fus sys}
        Let $S$ be a finite $p$-group. 
        \begin{enumerate}
            \item[(a)] When $S\in\syl_p(G)$ for a finite group $G$, then the \emph{fusion system of $G$ over $S$} is the category whose objects are the subgroups of $S$ and whose morphism sets are $\hom_G(P,Q)$. We denote the fusion system of $G$ over $S$ as $\f_S(G)$.
            \item[(b)] A \emph{saturated fusion system over $S$} is a category $\f$ where the objects of $\f$ are the subgroups of $S$ and that satisfies the following two properties for all $P,Q\leq S$: 
		\begin{itemize}
			\item $\hom_S(P,Q)\subset\hom_{\f}(P,Q)\subset\inj(P,Q)$; and 
			\item each $\vp\in\hom_{\f}(P,Q)$ is the composite of an $\f$-isomorphism followed by inclusion.
		\end{itemize} 
        and it satisfies an additional two axioms (see Definition 2.2 in \cite{aschbacher2011fusion}).
        \item[(c)] A subgroup $P\leq S$ is \emph{$\f$-centric} if $C_S(Q)=Z(Q)$ for all $Q\in P^\f$. Equivalently, when $\f=\f_S(G)$, if $P$ is $p$-centric in $G$. 
        \end{enumerate}
	\end{defn}
	\noindent In particular, a fusion system of the form $\f_S(G)$ where $G$ is a finite group and $S\in\syl_p(G)$ is a saturated fusion system \cite[Theorem 3.7]{puig2009frobeniusversus}. \par 
	
	\begin{defn}[Definition 5.1 and 5.2 in \cite{craven2011theory}, Definition 1.3 in \cite{aschbacher2016fusion}]
		Let $\f$ and $\mathcal{E}$ be fusion systems on the finite $p$-groups $S$ and $P$, respectively. 
		\begin{enumerate}
			\item[(a)] A \emph{morphism of fusion systems} $\f\to\mathcal{E}$  consists of a group homomorphism $\varphi:S\to P$ and a functor $\Phi:\f\to\mathcal{E}$ where $\Phi(T)=\vp(T)$ and which satisfies $\Phi_{T,U}(\psi)\circ\varphi|_T=\vp|_U\circ\psi$ for each $\psi\in\hom_{\f}(T,U)$. We denote a morphism of fusion systems by the pair $(\vp,\Phi)$. 
			\item[(b)] We say that a $\Phi$ is \emph{surjective} if $\vp$ is surjective, and $\Phi$ is full. 
			\item[(c)] We say that $\Phi$ is an \emph{isomorphism} if $\Phi$ is full and $\vp$ is injective.
            \item[(d)] We say that a homomorphism $\vp:S\to P$ is \emph{fusion preserving} if $\vp$ extends to an isomorphism $\Phi:\f\to\mathcal{E}$.
            \item When two groups $G$ and $H$ have isomorphic fusion systems at $p$, we say that they have equivalent $p$-fusion, denoted by $G\sim_p H$
		\end{enumerate}
	\end{defn}
	
	\begin{defn}[Definition 4.1 in Part I of \cite{aschbacher2011fusion}]\label{def: sbgps of fus sys}
		Fix a prime $p$, a finite group $G$, and $S\in\syl_p(G)$. Let $\f=\f_S(G)$ and let $Q\leq S$. 
		\begin{itemize}
			\item We say that $Q$ is \emph{central in $\f$} if $Q\leq Z(S)$ and for every $P,R\leq S$ and all $\vp\in\hom_\f(P,R)$, there is an extension $\bar\vp\in\hom_\f(PQ,RQ)$ such that $\bar\vp|_Q=\id_Q$. 
			\item Let $Z(\f)\leq Z(S)$ denote the largest subgroup of $S$ which is central in $\f$.
		\end{itemize}
	\end{defn}
	One can interpret Glauberman's $Z^*$-theorem to say that for any $S\in\syl_2(G)$, we have $Z(\f_S(G))=Z^*(G)\cap S$ where $Z^*(G)$ is the subgroup of $G$ that contains $O_{2'}(G)$ and is such that $Z^*(G)/O_{2'}(G)=Z(G/O_{2'}(G))$  If $O_{2'}(G)\leq Z(G)$, we can strengthen this statement. 
	\begin{lem}\label{lem: zf is 2 part of zg}
		Let $ G $ be a finite group, $ S\in\syl_2(G) $, and $ \f=\f_S(G) $. If $ O_{2'}(G)\leq Z(G) $, then $Z(\f)=O_2(Z(G))$ and $ Z(G)=O_{2'}(G)Z(\f) $. 
	\end{lem}
	\begin{proof}
		Since $O_{2'}(G)=O_{2'}(Z^*(G))$, the quotient $Z^*(G)/O_{2'}(G)$ has no nonidentity normal subgroups of odd order. But, $Z^*(G)/O_{2'}(G)$ is abelian, so it is a 2-group. Since $Z(\f)\in\syl_2(Z^*(G))$, $Z(\f)$ maps isomorphically to $Z^*(G)/O_{2'}(G)$. Hence, since $|Z^*(G)|=|O_{2'}(G)||Z^*(G)/O_{2'}(G)|$, it follows that $Z^*(G)=Z(\f)O_{2'}(G)$. \par 
		In particular, since $O_{2'}(G)\leq Z(G)$
		\begin{align*}
			[Z^*(G),Z(\f)] = [O_{2'}(G)Z(\f),Z(\f)]=[O_{2'}(G),Z(\f)][Z(\f),Z(\f)]=1
		\end{align*}
		and $Z(\f)=O_2(Z^*(G))\leq O_2(G)$. Therefore, for any $z\in Z(\f)\leq O_2(G)$ and any $g\in G$, $gzg^{-1}\in O_2(G)\leq S$. Since $Z(\f)$ is $\f_S(G)$-central, $gzg^{-1}=z$. Hence, $Z(\f)\leq Z(G)$. 
	\end{proof}
	\noindent Lemma \ref{lem: zf is 2 part of zg} can be generalized to odd primes; however, this requires the generalization of the  Glauberman $Z^*$-theorem to odd primes which itself relies the classification of finite simple groups.

	Recall the following definitions regarding quotient fusion systems.
	\begin{defn}[Definition 4.1 in \cite{aschbacher2011fusion}]
	Let $\f$ be a fusion system on a finite $p$-group $S$.  A subgroup $T\leq S$ is said to be \emph{strongly closed} if no element of $T$ is $\f$-conjugate to an element of $S-T$. 
	\end{defn}
	\begin{defn}[Definition 5.5 in \cite{craven2011theory}, Definition 5.9 in \cite{aschbacher2011fusion}]\label{defn: quotient fusion system}
		Let $\f$ be a fusion system on a finite $p$-group $S$. Let $T\leq S$ be strongly closed in $\f$. Let $\vp:S\to S/T$ denote the canonical projection. The fusion system $\f/T$ is defined to be the category whose objects are the subgroups of $S/T$ and, for $P/T, Q/T\leq S/T$, 
		\begin{align*}
			\hom_{\f/T}(P/T,Q/T)=\{(xT\mapsto\psi(x)T)\in\hom(P/T,Q/T)\mid\psi\in\hom_\f(P,Q)\}.
		\end{align*}
		Moreover, there is a surjective fusion system morphism $(\vp,\Phi):\f\to\f/T$ induced by the group homomorphism $\vp:S\to S/T$. We denote the morphism $xT\mapsto\psi(x)T$ by $\overline{\psi}$.     
	\end{defn}
	\noindent If $\f$ is a saturated fusion system, then the quotient $\f/T$ is a saturated fusion system for any strongly closed subgroup $T\leq S$ (see Proposition 5.11 in \cite{craven2011theory} or Lemma 5.4 in \cite{aschbacher2011fusion}). \par 
	
	We are particularly interested in the quotient of a fusion system of the form $\f_S(G)$ by $Z(\f)$ when $O_{2'}(G)\leq Z(G)$. Towards describing this quotient, we have the following lemma:
	\begin{lem}[Lemma 8.8 in \cite{aschbacher_normal_2008}]\label{lem: description of f mod zf}
		Let $G$ be a finite group with $S\in\syl_p(G)$. If $N\norm G$, then $T=S\cap N$ is strongly closed in $\f_S(G)$. Furthermore, $\f_S(G)/T\cong\f_{S/T}(G/N)$ where $SN/N$ is identified with $S/T$. 
	\end{lem}
	\noindent If $O_{2'}(G)\leq Z(G)$, then by Lemma \ref{lem: zf is 2 part of zg}, $Z(\f)=Z(G)\cap S$. Using Lemma \ref{lem: description of f mod zf}, $\f_S(G)/Z(\f)\cong \f_{S/Z(\f)}(G/Z(G))$. As a result, when $O_{2'}(G)\leq Z(G)$, we will identify $\f_S(G)/Z(\f)$ with $\f_{S/Z(\f)}(G/Z(G))$.

	\subsection{Automorphisms of fusion systems}   
	The automorphism group of a fusion system $\f$ over a finite $p$-group $S$ is the subgroup of $\aut(S)$ that consists of the fusion preserving automorphisms of $S$. It is defined that 
	\begin{align*}
		\out(\f)=\aut(\f)/\aut_\f(S).
	\end{align*} 
	\par 
	
	\begin{lem}\label{lem: pif is well defined}
		Let $\f$ be a saturated fusion system over a finite $p$-group $S$. The center of $\f$ is $\aut(\f)$-invariant. The map $\a\mapsto\a|_{S/Z(\f)}$ determines a well-defined group homomorphism $\tilde{\pi}_\f:\aut(\f)\to\aut(\f/Z(\f))$ that sends $\aut_{\f}(S)$ to $\aut_{\f/Z(\f)}(S/Z(\f))$, and hence determines a well-defined homomorphism $\pi_\f:\out(\f)\to\out(\f/Z(\f))$. 
	\end{lem}
	\begin{proof}
		Let $(\pi,\Pi):\f\to\f/Z(\f)$ denote the fusion system morphism induced by the canonical projection $\pi:S\to S/Z(\f)$ given in Definition \ref{defn: quotient fusion system}. \par 
		
		Let $\a\in\aut(\f)$ and let $x\in Z(\f)$. Let $\vp\in\hom_{\f}(\langle\a(x)\rangle, S)$. Since $(\a,A)$ is a fusion system automorphism, there is a morphism $\overline{\vp}\in\hom_{\f}(\langle x\rangle,S)$ such that $\a\circ \overline{\vp} = \vp\circ \a|_x$. Since $x$ is $\f$-central, $x^\f=\{x\}$ and in particular, $\overline{\vp}(x)=x$. Therefore, 
		\begin{align*}
			\vp(\a(x)) = \a(\overline{\vp}(x)) = \a(x).
		\end{align*}
		Hence, $\a(x)\in Z(\f)$ and $\a(Z(\f))\leq Z(\f)$. Since $\a\in\aut(S)$ and $S$ is finite, equality holds. \par 
		
		We will show that $\a|_{S/Z(\f)}$ extends to an automorphism of $\f/Z(\f)$ and we will denote this automorphism by $(\a|_{S/Z(\f)},A|_{S/Z(\f)})$. Let $\vp\in\hom_{\f/Z(\f)}(T/Z(\f),U/Z(\f))$. By definition, $\vp$ is the restriction of some $\psi\in\hom_\f(T,U)$. Since $\a$ is a fusion preserving automorphism of $S$, there is some morphism $A(\psi)\in\hom_{\f}(\a(T),\a(U))$ such that $A(\psi)\circ\a|_T=\a|_U\circ\psi$. Defining $A|_{S/Z(\f)}(\vp)=\Pi(A(\psi))$, it follows that $(\a|_{S/Z(\f)},A|_{S/Z(\f)})$ is a fusion system morphism of $\f/Z(\f)$. \par  
		
		It is clear that $\a|_{S/Z(\f)}$ is bijective. Let $\vp\in\hom_{\f/Z(\f)}(T/Z(\f),U/Z(\f))$. Note that $\vp$ is the restriction of some $\psi\in\hom_\f(T,U)$. Since $\a$ is a fusion preserving automorphism, there are subgroups $T',U'\leq S$ and $\psi'\in\hom_{\f}(T',U')$ such that $\psi=A(\psi')$. Therefore, there are subgroups, namely $T'/Z(\f),U'/Z(\f)\leq S/Z(\f)$ and the restriction of $\psi'$, which we denote by $\vp'\in\hom_{\f/Z(\f)}(T'/Z(\f),U'/Z(\f))$, such that $\vp=A|_{S/Z(\f)}(\vp')$. \par
		
		Lastly, if $\a\in\aut_{\f}(S)$, then since $(\pi,\Pi)$ is a fusion system morphism,
	\begin{align*}
			\Pi(\a)\in\hom_{\f/Z(\f)}(S/Z(\f),S/Z(\f))=\aut_{\f/Z(\f)}(S/Z(\f)).
	\end{align*}
This completes the proof. 
	\end{proof}

	\subsection{Centric linking systems}
	If $G$ is a group and $H$ and $K$ are subgroups of $G$, then the transporter set $T_G(H,K)$ is
	\begin{align*}
		T_G(H,K)=\{g\in G\mid gHg^{-1}\leq K\}.
	\end{align*}
	\begin{defn}[Definition 3.1 in \cite{aschbacher2011fusion}]\label{def: trans and cent link}
		Let $G$ be a finite group and $S\in\syl_p(G)$.
		\begin{itemize}
			\item The \emph{transporter category} of $G$ over $S$ is the category $\ts_S(G)$ whose objects are subgroups of $S$ and, given $P,Q\leq S$, $\mor_{\ts_S(G)}(P,Q)=T_G(P,Q)$. Composition of morphisms is multiplication in $G$. 
			\item The \emph{centric linking system of $G$ over $S$} is the category $\cl_S^c(G)$ whose objects are the $p$-centric subgroups of $G$ contained in $S$, and whose morphism sets are given by 
			\begin{align*}
				\mor_{\cl_S^c(G)}(P,Q)=T_G(P,Q)/O_{p'}(C_G(P))
			\end{align*}
			for all $P,Q\leq S$ which are $p$-centric in $G$. Composition of morphisms is multiplication of cosets. 
		\end{itemize}
	\end{defn}
	\noindent For an arbitrary saturated fusion system $\f$ over a finite $p$-group $S$, the category $\ts_{\Delta}(S)$ is the category whose objects are $\Delta\subset\ob(\f)$ and whose morphism sets are give by $T_S(P,Q)$ for any $P,Q\in\Delta$.

	We now state the definition of a centric linking system associated to an arbitrary saturated fusion system. 
	\begin{defn}[Definition 3.1 in \cite{oliver2007extensions}]\label{def: linking system}
		Let $\f$ be a saturated fusion system over a $p$-group $S$. A \emph{transporter system associated to $\f$} is a finite category $\ts$ with object set $\Delta\subset\f$, together with a pair of functors
		\begin{align*}
			\ts_{\ob(\cl)}(S)\xrightarrow{\delta}\ts\xrightarrow{\pi}\f
		\end{align*}
		satisfying the following conditions:
		\begin{enumerate}
			\item[(A1)] $\Delta$ is a nonempty set that is closed under $\f$-conjugacy and passage to overgroups, $\d$ is the identity on objects, and $\pi$ is the inclusion on objects. 
			\item[(A2)] For each $P,Q\in\Delta$, the kernel
            \begin{align*}
                E(P):=\ker(\pi_{P,P}:\aut_{\ts}(P)\to\aut_{\f}(P))
            \end{align*}
            acts freely on $\mor_{\ts}(P,Q)$ right composition, and  $\pi_{P,Q}$ is the orbit map for this action. Here, $\aut_{\ts}(P)$ denotes $\mor_{\ts}(P,P)$. 
			\item[(B)] For each $P,Q\in\ob(\ts)$ and each $g\in T_S(P,Q)$, $\pi_{P,Q}$ sends $\d_{P,Q}(g)\in\mor_\cl(P,Q)$ to $c_g\in\hom_\f(P,Q)$, and $\d_{P,Q}:T_S(P,Q)\to\mor_{\ts}(P,Q)$ is injective. 
			\item[(C)] For all $\psi\in\mor_\ts(P,Q)$ and all $g\in P$, the diagram
			\[\begin{tikzcd}
				P && Q \\
				\\
				P && Q
				\arrow["\psi", from=1-1, to=1-3]
				\arrow["{\delta_{P,P}(g)}"', from=1-1, to=3-1]
				\arrow["{\delta_{Q,Q}(\pi(\psi)(g))}", from=1-3, to=3-3]
				\arrow["\psi"', from=3-1, to=3-3]
			\end{tikzcd}\]
			commutes in $\ts$.
            \item[(I)] $\d_{S,S}(S)$ is a Sylow $p$-subgroup of $\aut_{\ts}(S)$. 
            \item[(II)] Let $\vp\in\iso_{\ts}(P,Q)$, $P\norm\overline{P}\leq S$, and $Q\norm\overline{Q}\leq S$ be such that $\vp\circ\d_{P,P}(\overline{P})\circ\vp^{-1}\leq\d_{Q,Q}(\overline{Q})$. Then there is $\overline{\vp}\in\mor_{\ts}(\overline{P},\overline{Q})$ such that $\overline{\vp}\circ\d_{P,\overline{P}}(1)=\d_{Q,\overline{Q}}(1)\circ\vp$. 
		\end{enumerate}
	\end{defn}
	\noindent We say that a transporter system $\cl$ is a centric linking system associated to $\f$ if $\ob(\ts)=\f^c$. From here on, we adhere to the standard practice of abbreviating $\d_{P,P}$ to $\d_P$ and $\pi_{P,P}$ to $\pi_P$. \par 
	
	For a finite group $G$ with $S\in\syl_p(G)$, the centric linking system $\cl_S^c(G)$ satisfies the above definition for the saturated fusion system $\f_S(G)$. The functors $\d:\ts_{\f_S(G)^c}(S)\to\cl_S^c(G)$ and $\pi:\cl_S^c(G)\to\f_S(G)$ are the identity on objects and the inclusion on objects, respectively. For morphisms, the functor $\d$ maps $T_S(P,Q)\to T_G(P,Q)/O_{p'}(C_G(P))$ via projection for any $\f_S(G)$-centric subgroups $P,Q\leq S$. On morphisms, the functor maps $gO_{p'}(C_G(P))\in T_G(P,Q)/O_{p'}(C_G(P))$ to the conjugation map $c_g\in\hom_{\f_S(G)}(P,Q)$ for $P,Q\in\ob(\cl_S^c(G))$. \par

	\subsection{Automorphisms of centric linking systems}
	\begin{defn}[Definition 2.3 in \cite{glauberman2021rigid}]
		Let $\f$ and $\f'$ be saturated fusion systems, $\cl$ and $\cl'$ be associated centric linking systems, and $\ts=\ts_{\f^c}(S)$ and $\ts'=\ts_{(\f')^c}(S')$.
		\begin{enumerate}
			\item Let $\a:\cl\to\cl'$ be an equivalence of categories. It is said that 
			\begin{itemize}
				\item $\a$ is \emph{isotypical} if $\a(\d_P(P))=\d'_{\a(P)}(\a(P))$ for each subgroup $P\in\f^c$. 	
				\item $\a$ \emph{sends inclusions to inclusions} if $\a(\d_{P,Q}(1))=\d'_{\a(P),\a(Q)}(1)$ for each $P,Q\in\f^c$.
			\end{itemize}
			\item An \emph{isomorphism} is an equivalence $\cl\to\cl'$ that is isotypical and sends inclusions to inclusions. An \emph{automorphism} of a centric linking system $\cl$ is an isomorphism of $\cl$ to itself.
			\item An automorphism $\a:\cl\to\cl$ is said to be \emph{rigid} if $\a_S\circ\d_S=\d_S$ as homomorphisms $S\to\aut_{\cl}(S)$. 
			\item An automorphism $\a$ of $\cl$ is said to be \emph{inner} if there is some element $\g\in\aut_{\cl}(S)$ such that $\a$ is given on objects by $P\mapsto c_{\g}(P)$ where $c_{\g}(P)=\pi(\g)(P)$, and on morphisms by mapping $\psi:P\to Q$ to 
			\begin{align*}
				c_{\g}(\psi) = \g|_{Q,c_{\g}(Q)}\circ\psi\circ(\g|_{P,c_{\g}(P)})^{-1}
			\end{align*}
			where $\g|_{Q,c_{\g}(Q)}$ is the unique morphism from $Q$ to $c_{\g}(Q)$ in $\cl$ such that $\g\circ\d_{Q,S}(1)=\d_{c_{\g}(Q),S}(1)\circ\g|_{Q,c_{\g}(Q)}$. We refer to $c_\g$ as conjugation by $\g$. 
		\end{enumerate}
	\end{defn}
	Denote the group of automorphisms of a centric linking system by $\aut(\cl)$. Also, denote by $\out(\cl)$ the group
	\begin{align*}
		\out(\cl)=\aut(\cl)/\{c_{\g}\mid \g\in\aut_{\cl}(S)\}
	\end{align*}

	\subsection{Relating the outer automorphism groups \texorpdfstring{$\out(G)$}{Out(G)}, \texorpdfstring{$\out(\cl)$}{Out(L)}, and \texorpdfstring{$\out(\f)$}{Out(F)}}

    \begin{defn}[\cite{andersen2012reduced}]
        Let $S$ be a finite $p$-group. Let $\f$ be a saturated fusion system over $S$ with associated centric linking system $\cl$. 
        \begin{enumerate}
            \item[(a)] The map $\mu_{\cl}:\out(\cl)\to\out(\f)$ is the group homomorphism defined by $\mu_{\cl}([\a])=[\d^{-1}_S\a_S\d_S]$.
            \item[(b)] Let $G$ be a finite group such that $S\in\syl_p(G)$, and let $\f=\f_S(G)$ and $\cl=\cl_S^c(G)$. The map $\kappa_G:\out(G)\to\out(\cl)$ is the group homomorphism that sends the class $[\a]\in\out(G)$, for $\a\in N_{\aut(G)}(S)$, to the class $[\b]\in\out(\cl)$ where $\b(P)=\a(P)$ for an object $P$, and sends $gO_{p'}(C_G(P))\in\mor_{\cl}(P,Q)$ to $\b(g)O_{p'}(C_G(\b(P))))$.
            \item[(c)] Let $G$ be a finite group such that $S\in\syl_p(G)$, and let $\f=\f_S(G)$ and $\cl=\cl_S^c(G)$. Set $\overline{\kappa}_G=\mu_{\cl}\circ\kappa_G$ be the homomorphism induced by restriction to $S$.
        \end{enumerate}
    \end{defn}
    The following result of \cite{andersen2012reduced}, that was later refined by \cite{glauberman2021rigid}, characterizes the kernel of $\mu_{\cl}$:
	\begin{theo}[Lemma 1.16 in \cite{andersen2012reduced}, Theorem 1.2 in \cite{glauberman2021rigid}]\label{theo: kermu is a 2 group}
		For any centric linking system associated to a saturated fusion system $\f$ over a finite $p$-group, $\ker(\mu_{\cl})$ is trivial if $p$ is odd and an elementary abelian 2-group if $p=2$. 
	\end{theo}
	The kernel of $\mu_{\cl}$ is connected to higher derived functors of the center functor. In particular, from results of Chermak in \cite{chermak2013fusion}, Oliver in \cite{oliver2013existence}, and Glauberman and Lynd in \cite{glauberman2016control}, the following result holds, as refined by \cite{glauberman2021rigid}:
	\begin{prop}[\cite{chermak2013fusion},\cite{oliver2013existence},\cite{glauberman2016control},\cite{glauberman2021rigid}] \label{prop: lim1 is there but lim2 gone}
		Let $\f$ be a saturated fusion system over a finite $p$-group $S$. Then $\varprojlim_{\mathcal{O}(\mathcal{F}^c)}^k(\mathcal{Z}_\f)=1$ for all $k\geq2$. If $\cl$ is a centric linking system for $\f$, then there is an exact sequence
		\[\begin{tikzcd}
			1 && \varprojlim_{\mathcal{O}(\mathcal{F}^c)}^1(\mathcal{Z}_{\mathcal{F}}) && {\out(\mathcal{\cl})} && {\out(\f)} && 1
			\arrow[from=1-1, to=1-3]
			\arrow["{\lambda_\cl}", from=1-3, to=1-5]
			\arrow["{\mu_\cl}", from=1-5, to=1-7]
			\arrow[from=1-7, to=1-9]
		\end{tikzcd}\]
		Furthermore, $\varprojlim_{\mathcal{O}(\mathcal{F}^c)}^1(\mathcal{Z}_{\mathcal{F}})=1$ if $p$ is odd, and is an elementary abelian 2-group if $p=2$. 
	\end{prop}
	\noindent This proposition tells us that for any saturated fusion system, there exists a unique centric linking system up to rigid isomorphism as obstructions to the existence of centric linking systems would be detected by $\varprojlim_{\mathcal{O}(\mathcal{F}^c)}^2(\mathcal{Z}_\f)$. Due to the exact sequence, it is know that, regardless of the choice of prime, the map $\mu_\cl$ is surjective. As a direct result of Proposition \ref{prop: lim1 is there but lim2 gone}, it holds that $\varprojlim_{\mathcal{O}(\mathcal{F}^c)}^1(\mathcal{Z}_\f)\cong\ker(\mu_\cl)$. \par 

The following result of Glauberman and Lynd gives a description of $\ker(\k_G)$ when $O_{p'}(G)=1$. 
	\begin{theo}[Theorem 5.1 from \cite{glauberman2021rigid}]\label{theo: ker kappaG is a p prime group}
		Fix a prime $p$, a finite group $G$, and a Sylow $p$-subgroup $S$ of $G$. Let $\cl$ be the centric linking system for $G$. If $O_{p'}(G)=1$, then $\ker(\kappa_G)$ is a $p'$-group.
	\end{theo} 
	When we have a fusion system of the form $\f_S(G)$ (with centric linking system $\cl_S^c(G)$), it is said that $\f_S(G)$ is \emph{tamely realized} by $G$ if $\k_G$ is split surjective.
	\begin{defn}[Definition 1.3 in \cite{broto2019automorphisms}]\label{def: tameness}
		For a finite group $G$ and $S\in\syl_p(G)$, the fusion system $\f_S(G)$ is \emph{tame} if there is a finite group $G^*$ which satisfies the following two conditions:
		\begin{itemize}
			\item there is a fusion preserving isomorphism $S\cong S^*$ for some $S^*\in\syl_p(G^*)$; and
			\item the homomorphism $\k_{G^*}:\out(G^*)\to\out(\cl_S^c(G^*))$ is split surjective (i.e., $G^*$ tamely realizes $\f_{S^*}(G^*)$).
		\end{itemize}
	\end{defn}

	\subsection{Constrained fusion systems}
	For an arbitrary saturated fusion system $\f$ over a finite $p$-group $S$, it is said that $Q\norm S$ is normal in $\f$ if each $\a\in\hom_{\f}(T,U)$ extends to a morphism $\bar{\a}\in\hom_{\f}(TQ,UQ)$ which, when restricted to $Q$, maps $Q$ to itself. The saturated fusion system $\f$ is said to be \emph{constrained} if there is some $Q\norm S$ such that $Q$ is $\f$-centric and normal in $\f$. Due to the following result of \cite{broto2005subgroup}, the centric linking systems of constrained fusion systems have no noninner rigid automorphisms.
	\begin{prop}[Proposition 4.2 in \cite{broto2005subgroup}]\label{prop: constrained fs have no noninner rigid autos}
		Let $\f$ be any constrained saturated fusion system over a finite $p$-group $S$. Then 
		\begin{align*}
			{\varprojlim_{\mathcal{O}(\mathcal{F}^c)}}^k(\mathcal{Z}_{\mathcal{F}})=1 \text{ for all }k\geq1.
		\end{align*}
		In particular, there is a centric linking system associated to $\f$ which is unique up to rigid isomorphism. Furthermore, the centric linking system $\cl$ has no noninner rigid automorphisms.
	\end{prop}
	\noindent In particular, when $\f=\f_S(G)$ for some Sylow $p$-subgroup $S$, and $S$ is abelian, then the fusion system $\f_S(G)$ is constrained and the centric linking system $\cl_S^c(G)$ has no noninner rigid automorphisms. \par 
	
	\subsection{The known quasisimple fusion systems at the prime 2}
	Let $\f$ be a saturated fusion system over a finite $p$-group $S$. A subsystem $\mathcal{E}$ of $\f$ over $T\leq S$ is said to be normal in $\f$ if $\mathcal{E}$ is saturated, $T$ is strongly closed, $\aut_{\f}(T)\leq\aut(\mathcal{E})$, and other technical conditions hold (see Definition 6.1 in \cite{aschbacher2011fusion}). It is said that $\f$ is simple if it contains no proper nontrivial normal fusion subsystems. Following Corollary 7.5 in \cite{aschbacher2016fusion}, we say that $\f$ is quasisimple if $\mathfrak{foc}(\f)=S$ and $\f/Z(\f)$ is simple. The only known family of nonrealizable quasisimple saturated fusion sytems at the prime 2 are the Benson-Solomon systems. The other known quasisimple saturated fusion sytems at the prime 2 are of the form $\f_S(G)$ for a finite group $G$ and $S\in\syl_2(G)$.
	
	Let $G$ be a nonabelian finite simple group and fix a prime $p$ such that $p$ divides~$|G|$. We say that $G$ is a $p$-Goldschmidt group if for $S\in\syl_p(G)$,  $\f_S(G)=\f_S(N_G(S))$. The following result classifies all $2$-Goldschmidt groups. 
	\begin{theo}[\cite{goldschmidt19742}]\label{theo: what groups are 2gold}
		Let $G$ be a nonabelian finite simple group and $S\in\syl_2(G)$. Then $G$ is $2$-Goldschmidt if and only if one of the following holds:
		\begin{enumerate}
			\item $S$ is abelian.
			\item $G$ is of Lie type in characteristic $2$ of Lie rank $1$. 
		\end{enumerate}
	\end{theo}
	\noindent The finite simple groups that satisfy either 1 or 2 were classified by Goldschmidt in \cite{goldschmidt19742}. There are three classes of simple groups that have abelian Sylow 2-subgroups: $\PSL_2(q)$ with $q\equiv\pm3\pmod{8}$, ${}^2G_2(3^{2n+1})$ with $n\geq0$, and $J_1$. There are also three classes of simple groups that are of (twisted) Lie rank 1 in characteristic 2: $A_1(2^n)$ with $n\geq2$, ${}^2A_2(2^n)$ with $n\geq2$, and ${}^2B_2(2^{n+1})$ with $n\geq1$. \par 
	
	The class of known simple fusion systems at the prime 2 contains the fusion systems that are of the form $\f_S(G)$ where $G$ is a simple group that is \textit{not} 2-Goldschmidt. This is a result of Aschbacher in \cite{aschbacher2021quaternion}.
	\begin{theo}[Theorem 5.6.18 in \cite{aschbacher2021quaternion}]\label{theo: 2gold not simple}
		Let $G$ be a known simple group that is not 2-Goldschmidt. Then the 2-fusion system of $G$ is simple.
	\end{theo}
	\noindent Therefore, the class of known quasisimple fusion systems at the prime 2 consists of the fusion systems over the known simple groups that are not 2-Goldschmidt, the fusion systems that are of the form $\f_S(G)$ where $G$ is a quasisimple covering of a finite simple group that is not 2-Goldschmidt, and the Benson-Solomon fusion systems. \par

	In section 3, we will define $\mathfrak{Lie}$ to be the class of all quasisimple groups of Lie type (see Definition 3.1). Per Theorem \ref{theo: what groups are 2gold} and Theorem \ref{theo: 2gold not simple}, we need to remove the groups that are 2-Goldschmidt. We define
	\begin{align*}
		\mathfrak{Lie}^* = \mathfrak{Lie} - &\{\PSL_2(q)\mid q\equiv\pm3\pmod{8}\}- \{SL_2(q)\mid q\equiv\pm3\pmod{8}\} - \\ 
		&\{{}^2G_2(q)\mid q = 3^{2a+1},\text{ }a\geq1\}- \{{}^2G_2(3)'\} - \\
		&\{L_2(2^n)\mid n\geq2\} - \{U_3(2^n)\mid n\geq2\} - \{{ }^2B_2(2^{2n+1})\mid n\geq1\} 
	\end{align*}
	
	\indent Now, let $\mathfrak{Alt}=\{\alter(n)\mid n\geq6\}$, and $\mathfrak{Spor}$ be the collection of sporadic simple groups. We then set $\mathfrak{Spor}^*=\mathfrak{Spor}-\{J_1\}$. We make the following definition:
	\begin{defn}\label{def: known quasisimple saturated fusion systems at the prime 2}
		We say that $\f$ is a \emph{known quasisimple saturated fusion system at the prime 2} if $\f$ is a Benson-Solomon saturated fusion system or if $\f\cong\f_S(G)$ where 
		\begin{align*}
			G \in \mathfrak{Alt}\cup\mathfrak{Spor}^*\cup\mathfrak{Lie}^*
		\end{align*}
		and $S\in\syl_2(G)$, or if $G$ is a quasisimple cover of some group $H\in\mathfrak{Alt}\cup\mathfrak{Spor}^*\cup\mathfrak{Lie}^* $
	\end{defn}

	%___________________________________________________________	
	
	%-----------------------------------------------------%
	% 3. Background on fusion systems and linking systems %
	%-----------------------------------------------------%	
	\section{Background on finite groups of Lie type}
	
	We now fix the notation and terminology that will be used when discussing the finite groups of Lie type. We will also list some results which will be used later in Sections 4 and 5. For a more in depth discussion about the groups of Lie type, we refer to \cite{gorenstein2005classification}. \par 
	
	\begin{defn}[Definition 1.7.1, 1.15.1, 2.2.1 in \cite{gorenstein2005classification}]
		Fix a prime $q_0$ and let $\G$ be a connected algebraic group over $\overline{\F}_{q_0}$. 
		\begin{enumerate}
			\item[(A)] We say that $\G$ is \emph{simple} if $[\G,\G]\neq1$, and all proper closed normal subgroups of $\G$ are finite and central. If $\G$ is simple, then it is of \emph{universal type} if it is simply connected, and of \emph{adjoint type} if $Z(\G)=1$. 
			\item[(B)] A \emph{Steinberg endomorphism} of a connected simple algebraic group $\G$ is a surjective algebraic endomorphism $\s\in\en(\G)$ whose fixed subgroup is finite.
			\item[(C)] A $\s$\emph{-setup} for a finite group $G$ is a pair $(\G,\s)$, where $\G$ is a simple algebraic group over $\overline{\F}_{q_0}$, and where $\s$ is a Steinberg endomorphism of $\G$ such that $G=O^{q_0'}(C_{\G}(\s))$. 
			\item[(D)] Let $\mathfrak{Lie}(q_0)$ denote the class of finite groups with $\sigma$-setup $(\G,\sigma)$ where $\G$ is simple and is defined in characteristic $q_0$, and let $\mathfrak{Lie}$ be the union of the classes $\mathfrak{Lie}(q_0)$ for all primes $q_0$. We say that $G$ is of universal type (similarly adjoint type) if $\G$ is of universal type (adjoint type). 
		\end{enumerate}
	\end{defn}
	\noindent Furthermore, when $G$ is of universal type, we denote it by $G^{\un}$ and when $G$ is of adjoint type, we denote it by $G^{\ad}$. In general, $C_{\G}(\s)=G\cdot C_{\T}(\s)$ where $\T$ is a $\s$-stable maximal torus of $\G$ (see Theorem 2.2.6 in \cite{gorenstein2005classification}). When $G$ is of universal type, $C_{\T}(\s)\leq G$, so $C_{\G}(\s)=G$. 
	
	\subsection{Root systems and the Weyl group}
	We now review root systems and the Weyl group associated to a group of Lie type. Let $ \E^m $ be an $ m $-dimensional Euclidean space with inner product $ (-,-) $. For any two elements $ a,b\in\E^m $ where $ b\neq0 $, we set $ \langle a,b\rangle=2(a,b)/(b,b) $ and denote by $r_a:\E^m\to\E^m$ the orthogonal reflection $b\mapsto b-\langle b,a\rangle a$ in the hyperplane $ a^\perp $. 
	\begin{defn}[\cite{humphreys2012introduction}]
		A \emph{root system} is a subset $\Sigma\subset\E^m$ such that 
		\begin{enumerate}
			\item[(R1)] $\Sigma$ is finite, spans $\E^m$, does not contain 0, and is $r_\a$-invariant for all $\a\in\Sigma$.
			\item[(R2)] (Reduced) For all $\a\in\Sigma$, $\R\a\cap\Sigma=\{\pm\a\}$.  
			\item[(R3)] (Crystallographic) For every $\a,\b\in\Sigma$, $\langle\b,\a\rangle\in\Z$.
		\end{enumerate}
	\end{defn}
	The Weyl group of the root system $ \Sigma $, denoted $ W $, is the subgroup of isometries of $ \E $ generated by the orthogonal reflections $ r_\a $ where $ \a\in\Sigma $. A base is a linearly independent subset $\Pi\subset\Sigma$ such that every $ \a\in\Sigma $ is either a nonnegative or nonpositive linear combination of the vectors in $ \Pi $. The Dynkin diagram associated to a root system $\Sigma$ is dependent on a base $\Pi$. In \cite{gorenstein2005classification} (and other sources), they classify the root systems and Dynkin diagrams for the simple algebraic groups. See tables 1 and 2 in the Appendix (or Section 1.8 in \cite{gorenstein2005classification}) for the complete list of the root systems and their associated Dynkin diagrams for connected semisimple algebraic groups. \par 
	
	Fix a connected simple algebraic group over $\overline{\F}_{q_0}$ and a maximal torus $\T\leq\G$. A root subgroup of $\G$ is a $\T$-normalized one parameter closed subgroup; that is, a root subgroup of $\G$ is a closed subgroup of $\G$ that is isomorphic to $(\overline{\F}_{q_0},+)$ and is normalized by $\T$. The roots of $\G$ are then defined to be the characters associated to some root subgroup of $\G$ and the set of roots in $\G$ is denoted $\Sigma_{\G}(\T)$, or just $\Sigma$ if the context is clear (see section 1.9 in \cite{gorenstein2005classification}). Due to Theorem 1.9.5 of \cite{gorenstein2005classification}, since any two maximal tori are $\G$-conjugate, the set $\Sigma_{\G}(\T)$ is independent of the choice of the maximal torus. Furthermore, $\Sigma_{\G}(\T)$ corresponds to some irreducible crystallographic root system $\Sigma$. \par

	\subsection{Chevalley relations and centers of algebraic groups}
	In this section, we fix the following notation: let $\G$ be a semisimple algebraic group and fix a maximal torus $\T$. Let $\Sigma=\Sigma_{\G}(\T)$ be the associated root system. We list some Chevalley relations below (for a more detailed list of Chevalley relations, see section 1.12 in \cite{gorenstein2005classification}): 
	\begin{theo}[Theorem 1.12.1 in \cite{gorenstein2005classification}] \label{thm: chev rel}
		The $\T$-root subgroups $\overline{X}_\a$ have parameterizations $x_\a(t)$ such that when we set 
		\begin{align*}
			n_\a(t)&=x_\a(t)x_{-\a}(-t^{-1})x_\a(t) \\
			h_\a(t)&=n_\a(1)^{-1}n_\a(t)            
		\end{align*}
		for any $\a\in\Sigma$, $t\in\overline{\F}_{q_0}^\times$, then $\T=\langle h_\a(t)\mid\a\in\Sigma, t\in\overline{\F}_{q_0}^\times\rangle$, and the following properties hold for any $\a,\b\in\Sigma$ and all $t,u\in\overline{\F}_{q_0}^\times$:
		\begin{enumerate}
			\item $[h_\a(t),h_\b(u)]=1$;
			\item $h_a(t)h_\a(u)=h_\a(tu)$; and
			\item Let $\Pi=\{\a_1,\a_2,\cdots,\a_n\} $ be a base in $\Sigma$. Then If $\G$ is universal, $\Pi_{i=1}^nh_{\a_i}(t_i)=1$ if and only if $t_i=1$ for all $1\leq i\leq n$. 
		\end{enumerate}
	\end{theo}
	\noindent In particular, part (2) of the theorem implies that for a given $\a\in\Sigma$, $h_\a(-)$ is a homomorphism $\overline{\F}_{q_0}^\times\to\T$. Also, parts (1) and (3) of the above theorem tells us that $ \T\cong h_{\a_1}(\overline{\F}_{q_0}^\times)\times h_{\a_2}(\overline{\F}_{q_0}^\times)\times\cdots\times h_{\a_n}(\overline{\F}_{q_0}^\times) $ if $\G$ is of universal type. \par 
	
	From Lemma 2.4 in \cite{broto2019automorphisms}, the maximal torus $\T$ is centric in $\G$. Because of this, the generators of the centers of simple algebraic groups can be determined using the Chevalley relations in terms of elements of the chosen maximal torus. These generators are shown in Theorem 1.12.6 and Table 1.12.6 in \cite{gorenstein2005classification} (see Table 3 and Table 4 in the appendix).

	\subsection{Automorphisms of groups of Lie type}
	Our main goal is to relate the automorphism group of a finite group of Lie type $ G $ to the automorphisms group of the fusion system of $ G $. In order to do this, we need to fix notation for certain subgroups of $ \aut(G) $ when $G$ is a finite group of Lie type. The following notation is standard and follows \cite{gorenstein2005classification} and \cite{broto2019automorphisms}. 
	\begin{defn}[Definition 3.1 in \cite{broto2019automorphisms}] \label{def: auts of lie gps}
		Let $ \G $ be a simple algebraic group over $\overline{\F}_{q_0}$. Let $\Sigma$ be the crystallographic root system associated to $\G$ with respect to some choice of maximal torus and fix a base $\Pi\subset\Sigma$. 
		\begin{itemize}
			\item Let $ r $ be any power of $ q_0 $ and let $ \psi_r\in\en(\G) $ be the field endomorphism defined by $ \psi_r(x_\a(u))=x_\a(u^r) $ for each $ \a\in\Sigma $ and $ u\in\overline{\F}_{q_0} $. We then set $ \Phi_{\G}=\langle\psi_{q_0}\rangle $ to be the monoid of all field endomorphisms of $ \G $. 
			\item Let $ \Gamma_{\G} $ be the group of all graph automorphisms of $ \G $ as defined in Definition 1.15.5 in \cite{gorenstein2005classification}. 
			\item A Steinberg endomorphism $\sigma$ of $\G$ is \emph{standard} if $\sigma = \psi_q\circ\gamma=\gamma\circ\psi_q$, where $q$ is a power of $q_0$ and $\g\in\Gamma_{\G}$. A $\sigma$-setup $(\G,\sigma)$ for a finite subgroup $G<\G$ is standard if $\sigma$ is standard.
		\end{itemize}
	\end{defn}
	\noindent Let $\G$ be a simple algebraic group with root system $\Sigma$. Fix a base $\Pi\subset\Sigma$. When $(\G,q_0)\neq(B_2(\overline{\F}_{q_0}),2),(G_2(\overline{\F}_{q_0}),3)$, or $(F_4(\overline{\F}_{q_0}),2)$, then $\Gamma_{\G}$ is the group of all $\gamma\in\aut(\G)$ of the form $\gamma(x_{\a}(u))=x_{\rho(\a)}(u)$ for all $\a\in\Pi$ and $u\in\overline{\F}_{q_0}$, where $\rho$ is an isometry of $\Sigma$ that leaves $\Pi$ invariant. For more information regarding graph automorphisms, see Theorem 1.15.2 in \cite{gorenstein2005classification}. \par 
	
	Fix a prime $q_0$ and let $ G\in\mathfrak{Lie}(q_0) $ with standard $ \s $-setup $ (\G,\s) $ such that $\s=\psi_{q}\g$ with $\g\in\Gamma_{\G}$.  If $ \g=\id $, we say that $ G $ is a Chevalley group, and if $ \g\neq\id $, we say that $ G $ is a twisted group. Further, if $ G $ is a twisted group, but is not a Suzuki or a Ree group, then we say that $ G $ is a Steinberg group. The structure of the root systems and the centers of the twisted groups is explored in depth in section 2 of \cite{gorenstein2005classification} (or see Table 5). 
	
	We now fix the following notation for particular subgroups of the automorphism group of a finite group of Lie type. 
	\begin{defn}[Definition 3.3 in \cite{broto2019automorphisms}]\label{def: outdiag and inndiag def}
		Let $G\in\mathfrak{Lie}(q_0)$ for some prime $q_0$ with standard $\s$-setup $(\G,\s)$ and let $\T$ be a maximal torus of $\G$. Let $\aut_{\T}(G)$ be the subgroup of automorphisms of $G$ induced by conjugation by elements of $N_{\T}(G)$. Fix the following notation:
		\begin{align*}
			\innd(G)&=\aut_{\T}(G)\inn(G), \\
			\od(G)&=\innd(G)/\inn(G), \\
			\Phi_G&=\{\psi_{q}|_G  \mid q=q_0^a,\text{ } a\geq1\}.  
		\end{align*}
		Further, if $G$ is a Chevalley group, set $\Gamma_G=\{\g|_G\mid\g\in\Gamma_{\G}\}$ to be the group of graph automorphisms of $G$. If $G$ is a twisted group, we set $\Gamma_G$ to be the trivial group. 
	\end{defn}
	\noindent The following theorem of Steinberg explains how $\aut(G)$ decomposes in terms of $\innd(G)$, $\Phi_G$, and $\Gamma_G$ for a finite group of Lie type $G$. 
	\begin{theo}[\cite{steinberg1960automorphisms}, Section 3]\label{thm: autG=ind(G)P(G)G(G)}
		Let $G$ be a finite group of Lie type with $\s$-setup $(\G,\s)$ where $\G$ is of universal or adjoint form. Then $\aut(G)=\innd(G)\Phi_G\Gamma_G$ where $\innd(G)\norm\aut(G)$ and $\innd(G)\cap\Phi_G\Gamma_G=1$. 
	\end{theo}

	\subsection{Notation and hypotheses}
	We introduce the following notation that will be used throughout most of this work. We are following the notation and hypotheses set forth in \cite{broto2019automorphisms}. 
	\begin{no}\label{not: gp of lie type}
		Let $(\G,\s)$ be a $\s$-setup for the finite group $G$, where $\G$ is a connected, semisimple algebraic group of universal type over $\overline{\F}_{q_0}$ for an odd prime $q_0$. 
		\begin{enumerate}
			\item[(A)] Fix a maximal torus $\T\leq\G$ such that $\T^\s=\T$. Let $W=N_{\G}(\T)/\T$ be the Weyl group of $\G$.
			
			\item[(B)] Let $\Sigma$ be the set of all roots of $\G$ with respect to $\T$. Let $\Pi\subset\Sigma$ be a fixed base. For each $\a\in\Sigma$, let $r_{\a}\in W$ be the reflection in the hyperplane $\a^{\perp}$. Let $\Sigma^+\subset\Sigma$ be the set of positive roots with respect to $\Pi$. For each $\a\in\Sigma^+$, let $\height(\a)$ denote the height of $\a$: the number of summands in the decomposition of $\a$ as a sum of simple roots. 
			
			\item[(C)] For each $\a\in\Sigma$, let $\overline{X}_\a<\G$ denote the root group for $\a$. That is, $\overline{X}_\a=\{x_\a(\l)\mid\l\in\overline{\F}_{q_0}\}$ with respect to some fixed Chevalley parameterization of~$\G$.
			
			\item[(D)] Set $T=\T\cap G=C_{\T}(\s)$. 
			
			\item[(E)] Let $\t\in\aut(V)$ and $\rho\in\aut(\Sigma)$ be the orthogonal automorphism and permutation, respectively, such that for each $\a\in\Sigma$, $\s(\bar{X}_\a)=\bar{X}_{\rho(\a)}$ and $\rho(\a)$ is a positive multiple of $\t(\a)$. Set $W_0=C_W(\tau)$. If $\rho(\Pi)=\Pi$, then set $V_0=C_V(\t)$, and let $\pro_{V_0}^\perp$ be the orthogonal projection of $V$ onto $V_0$. Let $\hat\Sigma$ be the set of equivalence classes in $\Sigma$ determined by $\t$, where $\a,\b\in\Sigma$ are equivalent if $\pro_{V_0}^\perp(\a)$ is a positive scalar multiple of $\pro_{V_0}^\perp(\b)$. Let $\hat\Pi\subset\hat\Sigma^+$ denote the images in $\hat\Sigma$ of $\Pi\subset\Sigma^+$. 
			
			\item[(F)] For $\a\in\Pi$ and $\lambda\in\overline{\F}_{q_0}^\times$, let $\hat{h}_\a(\lambda)\in T$ be an element in $G\cap\langle h_\b(\overline{\F}_{q_0}^\times)\mid\b\in\hat\a\rangle$ whose component in $h_\a(\overline{\F}_{q_0}^\times)$ is $h_\a(\lambda)$ (if there is such an element). 
		\end{enumerate}
	\end{no}
	We now fix the following hypotheses. 
	\begin{hypth}\label{hyp: gp of lie type}
		Following Notation \ref{not: gp of lie type}, 
		\begin{enumerate}
			\item[(I)]Suppose $\s=\vp_q\circ\g=\g\circ\vp_q$ where 
			\begin{itemize}
				\item $q$ is a power of the odd prime $q_0$ with $q_0\equiv1\pmod{4}$, that is, $q=q_0^b$ with $b=2^l$ and $l\geq0$. Set $k=l+2$. 
				\item $\vp_q\in\Phi_{\G}$ is a field automorphism and $\g\in\Gamma_{\G}$ is a graph automorphism such that $|\g|\leq2$.
			\end{itemize}
			\item[(II)] The algebraic group $\G$ is of universal type and $N_G(T)$ contains a nonabelian Sylow $2$-subgroup of $G$. Set $A=O_2(T)$, and fix $S\in\syl_2(N_G(T))\subset\syl_2(G)$ so that $A\leq S$.
		\end{enumerate}
	\end{hypth}	
	
	The 2-part of a positive integer $n$, denoted $n_2$, is the largest power of 2 that divides $n$. 
	\begin{lem}\label{lem: 2-part of 5}
		Assume Hypotheses \ref{hyp: gp of lie type}. Then $(q-1)_2=2^{l+2}=2^k$.
	\end{lem}
	\begin{proof}
		Recall that $q=q_0^{2^l}$, $q\equiv1\pmod{4}$ by Hypotheses \ref{hyp: gp of lie type}. When $ l=0 $, then $ q_0-1\equiv0\pmod{4} $. Hence, $(q_0-1)_2\geq 2^2$ . By Lemma 1.13 in \cite{broto2019automorphisms}, we know that  
		\begin{align*}
			(q-1)_2&=(q_0^{2^l}-1)_2 \\
			&=(q_0-1)_2 2^l \\
			&=2^2 2^l \\
			&=2^{l+2}
		\end{align*}
		This finishes the proof. 
	\end{proof}
	\noindent If $\G$ is of type $A_1$, then in order to satisfy (II) of Hypotheses \ref{hyp: gp of lie type}, $q_0\equiv1\pmod{8}$ and, as a result, $k\geq3$ in this case. \par
	
	The notation and hypotheses follow \cite{broto2019automorphisms} and the following proposition justifies the hypotheses. 
	\begin{prop}[Proposition 6.2 in \cite{broto2019automorphisms}]\label{prop: justification}
		Assume that $G$ is a group of Lie type that is of universal type over a field of odd prime power order $q$. Fix $S\in\syl_2(G)$, and assume that $S$ is nonabelian. Then there is an odd prime $q^*$, a group $G^*$ of universal type over a field of odd prime power $q^*$ order, and $S^*\in\syl_2(G^*)$, such that $\f_S(G)\cong\f_{S^*}(G^*)$, and $G^*$ has a $\s$-setup which satisfies Hypotheses \ref{hyp: gp of lie type}. Moreover, if $G^*\cong G_2(q^*)$ is a power of a $q_0^*$, then we can arrange that either $q^*=5$ or $q_0^*=3$. 
    \end{prop}

	%___________________________________________________________	
	
	%--------------------------------------------%
	% 4. Isometries of \Sigma acting on A/Z(\f)  %
	%--------------------------------------------%
	\section{Isometries of $\Sigma$ acting on $A/Z(\f)$}
	
	Throughout this section, let $\G$ be a simple and universal algebraic group over $\overline{\F}_{q_0}$ where $q_0$ is an odd prime. Let $\sigma$ be a Steinberg Endomorphism that satisfies Hypotheses \ref{hyp: gp of lie type}. Let $\T\leq\G$ be a $\sigma$-invariant maximal torus of $\G$ so that Hypotheses \ref{hyp: gp of lie type}(II) holds. Set $G=C_{\G}(\s)$ and $T=C_{\T}(\s)$ with $\s=\vp\circ\gamma=\gamma\circ\vp$ where $\vp\in\Phi_{\G}$ and $\g\in\Gamma_{\G}$ and let $A=O_2(T)$. Let $\Sigma$ be the reduced crystallographic root system associated to $G$ with base $\Pi$. Recall that by Lemma \ref{lem: zf is 2 part of zg}, $Z(\f)=O_2(Z(\f))$ so we can identify the quotient $AZ(G)/Z(G)$ with $A/Z(\f)$. \par 
		
	The aim of this section is to prove that isometries of a given root system act faithfully on $A/Z(\f)$ when Hypotheses \ref{hyp: gp of lie type} are satisfied. This is accomplished by proving an extension of a result of \cite{broto2019automorphisms} regarding how isometries of a given root system act on $A$. We extend their results to adjoint groups of Lie type in two ways. First, we show that the action of the Weyl group $W_0$ is faithful on $A/Z(\f)$ (assuming $q\equiv1\pmod{4}$ and $(q-1)_2\geq3$ when $G=A^{\un}_1(q)$). We then show that when $G$ is not $A^{\un}_1(q)$ or ${}^2D_n^{\un}(q)$ with $n\geq3$ (under the same assumptions), the group $W_0\Gamma_G\Phi_G$ acts faithfully on $A/Z(\f)$. \par

Recall that if $\g$ is an isometry of the root system and a simple root $\a\in\Pi$, then $\g(\a)$ is again a root. To show that the action of the isometry is faithful, we show that when the image of the root is written as a linear combination of simple roots, the corresponding subdiagram of the Dynkin diagram associated to the simple roots is a connected subdiagram of the Dynkin diagram associated to $\Pi$. Using this result regarding connected subdiagrams, we can show that in most cases $\g(\a)$ is not a root unless $\g(\a)=\a$ which will give us that the action of the group of isometries is faithful on the quotient $A/Z(\f)$. The proof of this lemma is carried out case by case because of the difficulty describing the generators of the center of the groups of Lie type in a uniform way in terms of the co-root lattice. 
	\subsection{Scalar automorphisms of \texorpdfstring{$A$}{A} and \texorpdfstring{$A/Z(\f)$}{A/Z(F)}}
	For any finite abelian group $B$, we follow \cite{broto2019automorphisms} and denote its ``scalar'' automorphisms by 
	\begin{align*}
		\aut_{\scal}(B) = \{\psi_k\mid(k,|B|)=1\}\leq Z(\aut(B))
	\end{align*} 
	where $\psi_k(b)=b^k$ for every $b\in B$. In particular, we have the two groups $\aut_{\scal}(A)$ and $\aut_{\scal}(A/Z(\f))$ when following Notation \ref{not: gp of lie type} and Hypotheses \ref{hyp: gp of lie type}. Set $\hat{\Phi}_G=\langle\vp_{q_0}|_G\rangle$ and note that by the following proposition and Hypotheses \ref{hyp: gp of lie type}, $\innd(G)\hat{\Phi}_G=\innd(G)\Phi_G$. 
	\begin{prop}[Proposition 3.6(d) in \cite{broto2019automorphisms}]
		Adapt Notation \ref{not: gp of lie type}, and let $\s$ be any Steinberg endomorphism of $\G$ and set $G=O^{q_o'}(C_{\G}(\s))$. If $\psi_{q_0}$ normalizes $G$, then $\innd(G)\langle\psi_{q_0}|_G\rangle=\innd(G)\Phi_G$.
	\end{prop}
	\noindent We can then use the following lemma of \cite{broto2019automorphisms} to see that $\hat{\Phi}_G$ maps to scalar automorphisms of $A$ under restriction. The following lemma is a combination of parts (b) and (c) of \cite[Lemma 5.12(b) and (c)]{broto2019automorphisms}.
	\begin{lem}
		Assume Hypotheses \ref{hyp: gp of lie type}. Set 
		\begin{align*}
			\aut(A,\f)=\{\b\in\aut(A)\mid\b=\overline{\b}|_A,\text{ some }\overline{\b}\in\aut(\f)\}.
		\end{align*}
		Let $\chi:\hat{\Phi}_G\to\aut(A,\f)$ be the homomorphism induced by restriction from $G$ to $A$. Then $\chi$ is injective and $\chi(\hat{\Phi}_G)$ has index 2 in $\aut_{\scal}(A)$. 
	\end{lem}
	
	Note that $\psi_k\in\aut_{\scal}(A)$ is the automorphisms of $A$ which raises elements of $A$ to the $k^{\text{th}}$ power for an odd integer $k$. Therefore, when restricted to the quotient $A/Z(\f)$, the automorphism $\psi_k|_{A/Z(\f)}$ is the automorphism that raises each element of $A/Z(\f)$ to the $k^{\text{th}}$ power. Since $A/Z(\f)$ is a 2-group and $k$ is odd, it follows that $\aut_{\scal}(A)$ maps injectively $\aut_{\scal}(A/Z(\f))$. Therefore, elements of $\Phi_G$, when restricted to $A/Z(\f)$, are elements of $\aut_{\scal}(A/Z(\f))$.

	\subsection{The co-root lattice}
	Let $G$ be a finite group of Lie type with standard $\s$-setup $(\G,\s)$ as in Notation \ref{not: gp of lie type} and Hypotheses \ref{hyp: gp of lie type}. Given any root $\a\in\Sigma$, we denote the corresponding co-root, or dual root, as $\a^\vee$ where $\a^\vee=2\a/(\a,\a)$, and we let $\Sigma^\vee$ be the collection of all co-roots. Note that for every reduced crystallographic root system, the co-root lattice is again a reduced crystallographic root system. \par 
	
	The root lattice or the co-root lattice is the collection of all $\Z$-linear combinations of elements of $\Sigma$ or $\Sigma^\vee$. The root lattice is denoted $\Z\Sigma$ and the co-root lattice is denoted $\Z\Sigma^\vee$. We state two results from \cite{broto2019automorphisms} relating the co-root lattice to a maximal torus of an universal algebraic group.
	\begin{lem}[Lemma 2.6 in \cite{broto2019automorphisms}] \label{lem: phi_lambda exists}   
		Adapt Notation \ref{not: gp of lie type} and assume that $\G$ is of universal type. 
		\begin{enumerate}
			\item[(a)] There is an isomorphism
			\begin{align*}
				\Phi:\Z\Sigma^\vee\otimes_{\Z}\overline{\F}_{q_0}^\times\to\T
			\end{align*}
			with the property that $\Phi(\a^\vee\otimes\lambda)=h_\a(\lambda)$ for each $\a\in\Sigma$ and each $\l\in\overline{\F}_{q_0}^\times$. 
		\end{enumerate}
		Fix some $\l\in\overline{\F}_{q_0}^\times$ and set $m=|\l|$. Set $\Phi_{\l}=\Phi(-,\l):\Z\Sigma^\vee\to\T$. 
		\begin{enumerate}
			\item[(b)] The map $\Phi_{\l}$ is $\Z[W]$-linear, $\ker(\Phi_{\l})=m\Z\Sigma^\vee$, and $\im(\Phi_{\l})=\{t\in\T\mid t^m=1\}$.  
		\end{enumerate}
	\end{lem}
	Lemma \ref{lem: phi_lambda exists} is going to be applied to $A\leq T$. For example, if $G$ is a Chevalley group, then since $\T\cong h_{\a_1}(\overline{\F}_{q_0}^\times)\times h_{\a_2}(\overline{\F}_{q_0}^\times)\times\cdots\times h_{\a_n}(\overline{\F}_{q_0}^\times)$ and $T=C_{\T}(\s)$, 
	\begin{align*}
		T\cong \langle h_{\a_1}(\l)\rangle\times \langle h_{\a_2}(\l)\rangle \times\cdots\times\langle h_{\a_n}(\l)\rangle
	\end{align*}
	where $\F_q^\times=\langle\l\rangle$. If we let $\mu\in\F_q^\times$ be such that $|\mu|=|\l|_2$, then we have an isomorphism 
	\begin{align*}
		\Phi_\mu:\frac{\Z\Sigma^\vee}{2^k\Z\Sigma^\vee}\to A.
	\end{align*} 
	Therefore, whenever we have an automorphism of $G$ that is induced from an isometry of the root system, the restriction to $A$ can be analyzed by passing to the co-root lattice. 
	\begin{lem}[Lemma 2.7 in \cite{broto2019automorphisms}]\label{lem: 27}
		Adapt Notation \ref{not: gp of lie type}, and assume also that $\G$ is of universal type. Let $\Gamma<\aut(V)$ be any finite group of isometries of $(V,\Sigma)$. Then there is an action of $\Gamma$ on $\T$, where $g(h_\a(u))=h_{g(\a)}(u)$ for each $\g\in\Gamma$, $\a\in\Sigma$, and $u\in\overline{\F}_{q_0}^\times$. Fix $m\geq3$ such that $q_0\nmid m$, and set $T_m=\{t\in\T\mid t^m=1\}$. Then $\Gamma$ acts faithfully on $T_m$. If $1\neq g$ and $l\in\Z$ are such that $g(t)=t^l$ for each $t\in T_m$, then $l\equiv-1\pmod{m}$.
	\end{lem}
	In particular, Lemma \ref{lem: 27} tells us that the group $\Gamma_G W_0$ acts faithfully on $A$. Furthermore, if there were some $\g\in \Gamma_G W_0$ such that $\g|_A\in\aut_{\scal}(A)$, then $\g$ acts as inversion or the identity.

\subsection{When \texorpdfstring{$G$}{G} is a Chevalley group}
	To prove the generalization of Lemma \ref{lem: 27}, we need the following result regarding the subdiagrams generated by subsets of simple roots. 
	\begin{defn}
		For any subset $X$ of $\Pi$, define $\Delta(X)$ to be the subdiagram of the Dynkin diagram induced by elements of $X$.
	\end{defn}	
	\begin{defn}
		Let $\b\in\Sigma$. The \emph{support of $\b$}, denoted $\supp(\b)$, is the subset $\supp(\b)\subset\Pi$ that contains each simple root that appears with nonzero coefficients when $\b$ is written as a sum of simple roots. 
	\end{defn}
	\begin{lem}[Corollary 3(a) in Chapter VI in \cite{bourbaki4to6}]\label{lem: sums of fund roots}
        Let $\a\in\Sigma$. Then $\Delta(\supp(\a))$ is connected.  
	\end{lem}
	In order to prove the generalization of Lemma \ref{lem: 27}, it will be helpful to consider the cases of Chevalley groups and Steinberg groups separately. We first show that, under the hypotheses, the Weyl group $W_0$ acts faithfully on $A/Z(\f)$ when $G=A^{\un}_1(q)=\SL_2(q)$. Although this is very easy to see directly, we provide an argument that hints at the general strategy in the other types. We cannot generalize Lemma \ref{lem: 27} in the case that $G=A^{\un}_1(q)$ further as a field automorphism of $G$ of order 2 acts trivially on $A/Z(\f)$.   
\begin{lem}\label{lem: gen 27 for a1}
		Assume Hypotheses \ref{hyp: gp of lie type}. If $G=A^{\un}_1(q)$, then $W_0$ acts faithfully on $A/Z(\f)$.
	\end{lem}
	\begin{proof}
		If $G=A^{\un}_1(q)$ with $q\equiv1\pmod{8}$, then $G$ does not have abelian Sylow 2-subgroups. By \cite[Theorem 1.12.6]{gorenstein2005classification}, $Z(G)=\langle h_{\a_1}(-1)\rangle$. Since $G$ is of universal type, 
		\begin{align*}
			T&=\langle h_{\a_1}(\lambda)\rangle & &\text{and} & A&=\langle h_{\a_1}(\mu)\rangle
		\end{align*}
		where $\F_q^\times=\langle\l\rangle$ and $|\mu|=|\l|_2$. Suppose that there is some $w\in W_0$ such that $c_w|_{A/Z(\f)}=\id_{A/Z(\f)}$. We then have that
		\begin{align*}
			c_w(h_{\a_1}(\mu))=h_{\a_1}(\mu)h_{\a_1}(-1)^i
		\end{align*}
		where $i=0$ or $1$. Suppose that $i\neq 0$. In the co-root lattice (via $\Phi_\mu$ from Lemma \ref{lem: phi_lambda exists}), this corresponds to
		\begin{align*}
			w(\a_1^\vee)&=\a_1^\vee+2^{k-1}\a_1^\vee
		\end{align*}
		Since $\Phi_\mu$ is $\Z[W_0]$-linear, for any $\b\in\Sigma^\vee$ and $v\in W_0$, we have that $v(\b)\in\Sigma^\vee$. For $w(\a_1^\vee)$ to be a root in $\Sigma_{A_1}^\vee=\Sigma_{A_1}$, we must have that $1+2^{k-1}\equiv\pm1\pmod{2^k}$ which implies that $2^{k-1}\equiv0\pmod{2^k}$ or that $k=2$ and neither are possible by Hypotheses \ref{hyp: gp of lie type}. Hence, $i=0$ and $c_w(h_{\a_1}(\mu))=h_{\a_1}(\mu)$. This shows that $W_0$ centralizes $A$. Applying Lemma \ref{lem: 27}, we know that the action of $W_0$ is faithful on $A$. Hence, $w=1$. 
	\end{proof}
	
	We now fully generalize Lemma \ref{lem: 27} to the rest of the Chevalley groups. 
	\begin{lem}\label{lem: gen27 for chevalley gps}
		Assume Hypotheses \ref{hyp: gp of lie type}. Let $G$ be a Chevalley group and suppose that $G\neq A^{\un}_1(q)$. Let $\Gamma=\Gamma_G W_0$. Then $\Gamma$ acts faithfully on $A/Z(\f)$. Furthermore, if there is some $1\neq g\in\Gamma$ such that $g|_{A/Z(\f)}\in\aut_{\scal}(A/Z(\f))$ so that $g(t)=t^l$ for each $t\in A/Z(\f)$, then $l\equiv-1\pmod{\exp(A)}$.
	\end{lem}
	\begin{proof}
		By Lemma \ref{lem: zf is 2 part of zg}, $O_2(Z(G)) = Z(\f)$. In particular, when 
    \begin{align*}
        G\in\{A_n^{\un}(q)\text{ with }n\text{ even}, E_6^{\un}(q),E_7^{\un}(q),F_4^{\un}(q),G_2^{\un}(q)\}
    \end{align*}
    $A/Z(\f) = A$, and the result follows immediately from Lemma \ref{lem: 27}. Thus, we may assume that $Z(\f)\neq 1$. \par 

    Since $G$ is of universal type of Lie rank $n$, we may write  
		\begin{align*}
			T&=\langle h_{\a_1}(\lambda)\rangle \times\cdots\times\langle  h_{\a_n}(\lambda)\rangle \\
			A&=\langle h_{\a_1}(\mu)\rangle \times\cdots\times\langle h_{\a_n}(\mu)\rangle
		\end{align*}
	where $\F_q^\times=\langle\l\rangle$ and $|\mu|=|\l|_2$. Write 
	\begin{align*}
		Z=\begin{cases}
			\langle h_{\a_1}(\mu)h_{\a_2}(\mu^2)\cdots h_{\a_n}(\mu^n)\rangle &\text{if }G= A^{\un}_n(q) \\
			Z(\f) &\text{else}
		\end{cases}
	\end{align*}
	so that $Z(\f)\leq Z$. Note that any automorphisms of $A$ that restrict to the identity on $A/Z(\f)$ must also restrict to the identity on $A/Z$. \par
		
	Suppose that $g\in\Gamma$, and $\psi\in\aut_{\scal}(A/Z(\f))$ are such that for every $a\in A$, $g(a)Z=\psi(a)Z$. In particular, for every $\a_i$, 
	\begin{align*}
		g(h_{\a_i}(\mu))=h_{\a_i}(\mu^{r})z_{\a_i}
	\end{align*}
	for some odd integer $r$ and $z_{\a_i}\in Z$. Passing to the co-root lattice via $\Phi_{\mu}$, this corresponds to 
	\begin{align*}
		g(\a_i^\vee)&=r\a_i^\vee+\d_i 
    \end{align*}
	where $\d_i\in\Z\Sigma^\vee$ is the appropriate linear combination relating to $z_{\a_i}$. We claim that $\d_i=0$ for all $i$. \\
    
    \noindent\textbf{Case $A_n^{\un}(q)$ with $n$ odd, $n\neq1$}: In this case, the elements of $Z$ correspond to vectors $(z_1,\ldots,z_n)$ with $z_j = mj$ for some $1\leq m\leq |\mu|$. For $g(\a_i^\vee)$ to be a root, all nonzero coefficients must be congruent to each other and, simultaneously, congruent to $\pm1$. When comparing the coefficients of $\a_1^\vee$, $\a_2^\vee$, and $\a_3^\vee$, this forces $m=0$, and hence $\d_i=0$. \\
    
    \noindent\textbf{Case $B^{\un}_n(q)$ with $n\geq2$}: Here, $Z$ contributes only to the terminal node of the Dynkin diagram. If $i$ is not adjacent to $n$ ($i<n-1$), then $\d_i$ introduces support disjoint from $\a_i^\vee$, so $g(\a_i^\vee)$ has disconnected support and is not a root by Lemma \ref{lem: sums of fund roots}. For $i=n-1$, checking the coefficient of $\a_n^\vee$ forces $\d_i=0$. \par 
    For $i=n$, we have
    \begin{align*}
        g(\alpha_n^\vee)=r\alpha_n^\vee + 2^{k-1}j\alpha_n^\vee,
    \end{align*}
    for $j=0$ or 1. From the previous cases, $\delta_k=0$ for $k<n$. Hence, $g$ acts by a scalar $r\equiv\pm1\pmod{2^k}$ on the $A_{n-1}$ subsystem. The simple co-root $\a_n^\vee$ is uniquely characterized by its bond to $\a_{n-1}^\vee$ on the $B_n$ Dynkin diagram. Any automorhism in $\Gamma$ must preserve this adjacency structure together with the induced scalar action on the $A_{n-1}$-subsystem. If $j=1$, then the image of $\a_n^\vee$ is shifted by a nontrivial power of $2$. This shift is not compatible with preserving the relationship between $\a_n^\vee$ and $\a_{n-1}^\vee$ under the scalar action determined on the $A_{n-1}$-chain. Therefore, $j=0$, and hence, $\delta_i=0$ for all $i$.\\
    
    \noindent\textbf{Case $C^{\un}_n(q)$ with $n\geq3$}: In this case, $Z$ contributes a sum of pairwise nonadjacent simple co-roots. Unless $n=3$ and $i=2$, any nonzero $\d_i$ yields disconnected support, and $g(\a_i^\vee)$ is not a root by \ref{lem: sums of fund roots}. For $n=3$ and $i=2$, analyzing the coefficients of $\a_1^\vee$ force $\d_2=0$. In any case, it follows that $\d_i=0$ for all $i$.\\ 
    
    \noindent\textbf{Case $D^{\un}_n$ with $n\geq4$}: Similar to the above case, $Z$ contributes a sum of pairwise nonadjacent simple co-roots. Unless, $n=4$ and $i=2$, applying Lemma \ref{lem: sums of fund roots} implies that $g(\a_i^\vee)$ is not a root unless $\d_i=0$. If $n=4$ and $i=2$, then analyzing the coefficients of $\a_1^\vee$ and $\a_3^\vee$ force $\d_2=0$. Hence, in all cases $\d_i=0$. \\
    
    \noindent\textbf{Case $E_7^{\un}(q)$}: The generator of $Z$ has support on non-adjacent nodes of the Dynkin diagram. Thus any nonzero $\delta_i$ yields disconnected support, so $g(\alpha_i^\vee)$ is not a root. Hence $\delta_i=0$. \\

    In all cases, we conclude that $z_{\a_i}=1$ for every $\a_i\in\Pi$, so $g$ acts as a scalar automorphism on $A$. Applying Lemma \ref{lem: 27}, if $g\neq1$, then the induced scalar automorphism must be inversion.
	\end{proof}

	\subsection{When \texorpdfstring{$G$}{G} is a Steinberg group}
	
	We first show that the Weyl group $W_0$ acts faithfully on $A/Z(\f)$ for  $G={}^2D^{\un}_n(q)$. We cannot generalize Lemma \ref{lem: 27} for ${}^2D^{\un}_n(q)$ further as there is an automorphism of~$G$ that is not inner that acts trivially on $A/Z(\f)$. We will discuss this automorphism further in Section 5. 
	
	\begin{lem}\label{lem: gen27 for 2dn}
		Assume Hypotheses \ref{hyp: gp of lie type}. If $G={}^2D^{\un}_n(q)$ for $n\geq3$, then $W_0$ acts faithfully on $A/Z(\f)$. 
	\end{lem}
	\begin{proof}	
		Let $\overline{G}=D_n(\overline{\F}_{q_0})$ and $\sigma=\gamma\vp_q$. From Table 5 (or Table 2.2 in  \cite{gorenstein2005classification}), $Z(G)=Z_{(q^n+1,4)}=Z_2$. Therefore, $Z(\f)=Z(G)=Z_2$. In particular, 
		\begin{align*}
			Z(G)=\langle h_{\a_{n-1}}(-1)h_{\a_n}(-1)\rangle
		\end{align*}
		since whether $n$ is even or odd, this is the only element of $Z(\G)$ that is fixed by $\sigma$.
		
		Note that
		\begin{align*}
			A=\langle h_{\a_1}(\mu^2)\rangle\times\langle h_{\a_2}(\mu^2)\rangle\times\cdots\times\langle h_{\a_{n-2}}(\mu^2)\rangle\times\langle h_{\a_{n-1}}(\mu)h_{\a_n}(\mu^q)\rangle
		\end{align*}
		where $\F^\times_{q^2}=\langle\lambda\rangle$ and $|\mu|=|\lambda|_2$. \par
		
		The full centralizer of $\g$ is $V_0=\text{Span}(\a_1,\a_2,\ldots,\a_{n-2},\a_{n-1}+\a_n)$. For every $1\leq i\leq n-1$, the projection of $\a_i$ onto $V_0$ is $\a_i$ since $\a_i$ is fixed by $\g$. The projection of $\a_{n-1}$ and $\a_n$ onto $V_0$ is $\frac{1}{2}(\a_{n-1}+\a_n)$ since $\a_{n-1}$ and $\a_n$ are orthogonal to each other. This gives us 
		\begin{align*}
			\hat{\Pi}=\left\{\a_1,\a_2,\ldots,\a_{n-2},\frac{1}{2}(\a_{n-1}+\a_n)\right\}=\{\hat{\a}_1,\hat{\a}_2,\ldots,\hat{\a}_{n-1}\}
		\end{align*}
		Hence,   
		\begin{align*}
			\hat{\a}_i^\vee=\begin{cases}
				\a_i^\vee, & \text{for }1\leq i<n-1; \\
				(\frac{1}{2}(\a_{n-1}+\a_{n}))^\vee, & \text{for }i=n-1
			\end{cases}
			=\begin{cases}
				\a_i^\vee, & \text{for }1\leq i<n-1; \\
				\a_{n-1}+\a_n, & \text{for } i=n-1.
			\end{cases}
		\end{align*}
		Also, we have $Z(\f)=\langle h_{\a_{n-1}}(-1)h_{\a_n}(-1)\rangle=\left\langle \hat{h}_{\a_{n-1}}(-1) \right\rangle$.
		
		\indent For all $i< n-2$, if $z_{\a_i}\neq1$, then we have in $\Z\Sigma_{{}^2D_n}^\vee$
		\begin{align*}
			w(2\hat{\a}_i^\vee)&=2\hat{\a}_i^\vee + 2^k\hat{\a}_{n-1}^\vee 
		\end{align*}
		which is not a root in $\Sigma_{{}^2D_n}^\vee=\Sigma_{B_{n-1}}^\vee=\Sigma_{C_{n-1}}$ by Lemma \ref{lem: sums of fund roots}. If $i=n-2$ and $z_{\a_i}\neq1$, then 
		\begin{align*}
			w(2\hat{\a}_{n-2}^\vee)&=2\hat{\a}_{n-2}^\vee + 2^k\hat{\a}_{n-1}^\vee 
		\end{align*}
		which is not a root in $\Sigma_{C_{n-1}}$ since $2^k\not\equiv\pm1\pmod{2^{k+1}}$. Finally, if $i=n-1$ and $z_{\a_i}\neq1$, then 
		\begin{align*}
			w(\hat{\a}_{n-1}^\vee)&=\hat{\a}_{n-1}^\vee + 2^k\hat{\a}_{n-1}^\vee = q\hat{\a}_{n-1} 
		\end{align*}
		which is only a root if $q\equiv\pm1\pmod{2^{k+1}}$, but $q\equiv2^k+1\pmod{2^{k+1}}$. Hence, $z_{\a_i}=1$ for all $i$. Therefore, we have that $c_w$ acts as the identity on $A$. Applying Lemma \ref{lem: 27}, $w=1$ which gives us that $W_0$ acts faithfully on $A/Z(\f)$. 
	\end{proof}
	
	We now fully generalize Lemma \ref{lem: 27} to the rest of the Steinberg groups.
	
	\begin{lem}\label{lem: gen27 for twisted gps}
		Assume Hypotheses \ref{hyp: gp of lie type}. Let $G$ be a Steinberg group and suppose that $G\neq {}^2D^{\un}_n(q)$ for $n\geq3$. Let $\Gamma=\Gamma_G W_0$. Then $\Gamma$ acts faithfully on $A/Z(\f)$. Furthermore, if there is some $1\neq g\in\Gamma$ such that $g|_{A/Z(\f)}\in\aut_{\scal}(A/Z(\f))$ so that $g(t)=t^l$ for each $t\in A/Z(\f)$, then $l\equiv-1\pmod{\exp(A)}$.
	\end{lem}
	\begin{proof}
		We only need to check ${}^2A^{\un}_n(q)$ for $n\geq2$ and $n\neq3$ since $O_2(Z({}^2E^{\un}_6(q)))=O_2(Z_{(3,q+1)})=1$ and $Z({}^2B^{\un}_2(q))=Z({}^2F^{\un}_4(q))=Z({}^2G^{\un}_2(q))=1$.  \par 
		
		Let $\overline{G}=A_n(\overline{\F}_{q_0})$ with $n\geq4$ and $\sigma=\gamma\vp_q$. From 5 (or Table 2.2 in  \cite{gorenstein2005classification}), $Z(G)=Z_{(q+1,n+1)}$. Therefore, if $n$ is even, since $ Z(\f)=O_2(Z(G)) $, it follows that $Z(\f)=1$. Hence, we can apply Lemma \ref{lem: 27}. \par 
		
		Now, suppose that $n\geq5$ and odd. Note that
		\begin{align*}
			A=\langle h_{\a_1}(\mu)h_{\a_n}(\mu^q)\rangle\times\langle h_{\a_2}(\mu)h_{\a_{n-1}}(\mu^q)\rangle\times\cdots\times\langle h_{\a_{m-1}}(\mu)h_{\a_{m+1}}(\mu^q)\rangle\times\langle h_{\a_m}(\mu^2)\rangle
		\end{align*}
		where $\F^\times_{q^2}=\langle\lambda\rangle$, $|\mu|=|\lambda|_2$, and $m=\frac{n+1}{2}$. From Table 4, it follows that $Z(\G)=\langle h_{\a_1}(\w)h_{\a_2}(\w^2)\cdots h_{\a_n}(\w^n)\rangle$ where $\w$ is a primitive $(n+1)^{\text{st}}$ root of unity. Therefore, 
		\begin{align*}
			Z(G)=Z(\f)=\langle h_{\a_1}(-1)h_{\a_3}(-1)\cdots h_{\a_n}(-1)\rangle
		\end{align*}
		
		Note that the full centralizer of $\g$ is $V_0=\text{Span}(\a_1+\a_n,\ldots,\a_{m-1}+\a_{m+1},\a_m)$. For every $i\neq m$, the projection of $\a_i$ onto $V_0$ is $\frac{1}{2}(\a_i+\a_{n-i+1})$ since $\a_i$ and $\a_{n-i+1}$ are orthogonal to each other. The projection of $\a_m$ onto $V_0$ is $\a_m$ since $\a_m$ is fixed by $\g$. This yields 
		\begin{align*}
			\hat{\Pi}=\left\{\frac{1}{2}(\a_1+\a_n),\frac{1}{2}(\a_1+\a_n),\ldots,\frac{1}{2}(\a_1+\a_n),\a_m\right\}=\{\hat{\a}_1,\hat{\a}_2,\ldots,\hat{\a}_m\}
		\end{align*}
		Hence, we have the following:
		\begin{align*}
			Z(\f)&=\langle h_{\a_1}(-1)h_{\a_3}(-1)\cdots h_{\a_n}(-1)\rangle \\
			&=\left\langle \hat{h}_{\a_1}(-1)\cdots \hat{h}_{\a_t}(-1) \mid k=2\left\lfloor\frac{m-1}{2}\right\rfloor+1 \right\rangle
		\end{align*}
		Also, 
		\begin{align*}
			\hat{\a}_i^\vee=\begin{cases}
				(\frac{1}{2}(\a_i+\a_{n-i+1}))^\vee, & \text{for }1\leq i<m; \\
				\a_m^\vee, & \text{for }i=m
			\end{cases}
			=\begin{cases}
				\a_i+\a_{n-i+1}, & \text{for }1\leq i<m; \\
				\a_m, & \text{for }i=m.
			\end{cases}
		\end{align*}
		Finally, note that $q\equiv2^k+1\pmod{2^{k+1}}$ and that $q+2^k\equiv1\pmod{2^{k+1}}$. \par 
		
		Suppose that $g\in\Gamma$, and $\psi\in\aut_{\scal}(A/Z(\f))$ are such that for every $a\in A$, $g(a)Z(\f)=\psi(a)Z(\f)$. In particular, for every $\hat{\a}_i$, 
		\begin{align*}
			g(\hat{h}_{\a_i}(\mu))=\hat{h}_{\a_i}(\mu^{r})z_{\a_i}
		\end{align*}
		for some odd $r$ and $z_{\a_i}\in Z(\f)$. In the co-root lattice, via $\Phi_{\mu}$ this equality corresponds to 
		\begin{align*}
            g(\hat{\a}_i^\vee)&=r\hat{\a}_i^\vee+2^kj\hat{\a}_1^\vee+\cdots+2^kj\hat{\a}_t^\vee 
		\end{align*}
		where $j=0$ or 1. \par 
		
		For all $i$, if $z_{\a_i}\neq1$, then we have in $\Z\Sigma_{{}^2A_n}^\vee$
		\begin{align*}
			g(\hat{\a}_i^\vee)&=r\hat{\a}_i^\vee+2^k\hat{\a}_1^\vee+2^k\hat{\a}_3^\vee+\cdots+2^k\hat{\a}_t^\vee \\
			&=\begin{cases}
				r\hat{\a}_i^\vee+2^k\hat{\a}_1^\vee+2^k\hat{\a}_3^\vee+\cdots+2^k\hat{\a}_t^\vee, & \text{for even }i; \\
				2^k\hat{\a}_1^\vee+2^k\hat{\a}_3^\vee+\cdots+(2^k+r)\hat{\a}_i+\cdots+2^k\hat{\a}_t^\vee, & \text{for odd }i.
			\end{cases} 
		\end{align*}
		Note that $n\geq5$ which forces $t\geq3$. Hence, we can use a similar argument to that of the $C_n$ case in the proof of Lemma \ref{lem: gen27 for chevalley gps} to see that $z_{\a_i}$ must be 1 for all $\a_i\in\Pi$. Applying Lemma \ref{lem: 27}, we get that if $g\neq1$, then $\psi$ is inversion.
	\end{proof}
	
	With all of the cases taken care of, the proof of Theorem \ref{lem: generalized 27} is an application of Lemma \ref{lem: gen 27 for a1}, Lemma \ref{lem: gen27 for chevalley gps}, Lemma \ref{lem: gen27 for 2dn}, and Lemma \ref{lem: gen27 for twisted gps}.

	\begin{theo}\label{lem: generalized 27}
		Assume Hypotheses \ref{hyp: gp of lie type}. The action of $W_0$ is faithful on $A/Z(\f)$. Further, let $\Gamma=\Gamma_G W_0$ and suppose that $G$ is not $A^{\un}_1(q)$ or ${}^2D^{\un}_n(q)$ with $n\geq3$. Then $\Gamma$ acts faithfully on $A/Z(\f)$. Furthermore, if there is some $1\neq g\in\Gamma$ such that $g|_{A/Z(\f)}\in\aut_{\scal}(A/Z(\f))$ so that $g(t)=t^l$ for each $t\in A/Z(\f)$, then $l\equiv-1\pmod{\exp(A)}$.
	\end{theo}

	%___________________________________________________________
	
	%--------------------------------------------------------------------------%
	% 5. Injectivity of $\mu_{\mathcal{L}}$ for most finite groups of Lie type %
	%--------------------------------------------------------------------------%
	\section{Injectivity of $\mu_{\mathcal{L}}$ for most finite groups of Lie type}

	In this section, we complete the classification of the finite groups of Lie type in odd characteristic whose centric linking systems have no noninner rigid automorphisms. In particular, we show that, up to equivalent 2-fusion, there are no noninner rigid automorphisms for the centric linking systems for the finite simple groups of Lie type with two families of exceptions: $A_1^{\ad}(q)$ with $q\equiv1\pmod{8}$ and ${}^2D^{\ad}_n(q)$ with $n\geq3$ and $q\equiv1\pmod{4}$. We do so by extending certain results of Broto, M{\o}ller, and Oliver in \cite{broto2019automorphisms}. \par 
	
	In \cite{broto2019automorphisms}, they show that the centric linking systems of the finite groups of Lie type in odd characteristic that are of universal type have no noninner rigid automorphisms (see Theorem \ref{theo: tameness and universals have no rigid autos} below). They analyze the map 
	\begin{align*}
		\kapc_G:\out(G)/O_{2'}(\od(G))\to\out(\f)
	\end{align*}
	and show that this map is an isomorphism at the prime 2 except for when $G = G_2(q)$ with $q_0=3$ and Oliver later shows in \cite[Appendix]{broto2019automorphisms} that the centric linking system for $G_2(q)$ does not have a noninner rigid automorphism where $q_0=3$. \par 
	
	Recall from section 2.7 and 2.5 that there are homomorphisms $\k_G:\out(G)\to\out(\cl)$ and $\mu_{\cl}:\out(\cl)\to\out(\f)$ whose composition is denoted by $\overline{\k}_G=\mu_{\cl}\circ\k_G$. The map $\kapc_G$ is $\overline{\k}_G$ restricted to the quotient $\out(G)/O_{2'}(\od(G))$. We compose the map $\kapc_G$ with the map $\pi_{\f}:\out(\f)\to\out(\f/Z(\f))$ (see Lemma \ref{lem: pif is well defined}) to obtain the homomorphism
	\begin{align*}
		\kapb_G:\out(G)/O_{2'}(\od(G))\to\out(\f/Z(\f)).
	\end{align*}
	We argue that under Hypotheses \ref{hyp: gp of lie type}, this map is injective for all cases except when $G/Z(G)=A_1^{\ad}(q)$ or $G/Z(G)={}^2D_n^{\ad}(q)$ with $n\geq3$. This tells us that the kernel of the map $\pi_\f\circ\bar{\kappa}_G$ is $O_{2'}(\od(G))$. Using the result of Glauberman and Lynd that states that the kernel of $\kappa_G$ is a $2'$-group (Theorem \ref{theo: ker kappaG is a p prime group}) along with a result regarding tameness (Theorem \ref{theo: tameness and universals have no rigid autos}), we can show that the map $\mu_\cl$ is injective in the nonexceptional cases. \par 
	
	For $G/Z(G)=A_1^{\ad}(q)$ or $G/Z(G)={}^2D_n^{\ad}(q)$ with $n\geq3$, we give explicit constructions of the noninner rigid automorphisms of the centric linking systems. When $G/Z(G)=A_1^{\ad}(q)$, the noninner rigid automorphism of the centric linking system which is induced by a field automorphism of $G=A_1^{\un}(q)$. When $G/Z(G)={}^2D^{\ad}_n(q)$ with $n\geq3$, the noninner rigid automorphism of the centric linking system is induced by an automorphism that is equivalent to a graph automorphism of the algebraic group. \par 
	
	To prove that the map $\kapb_G$ is injective, we first show that $A/Z(\f)$ is a 2-centric subgroup of $G/Z(G)$. We then use this to prove that $O_2(\od(G))$ maps injectively into a subgroup of $\out(\f/Z(\f))$ that consists of the classes of automorphisms which can be represented by automorphisms of $\f/Z(\f)$ that centralize $A/Z(\f)$. We then show that the kernel of the map $\pi_\f\bar{\kappa}_G$ is contained in $\od(G)$, and using the fact that $O_2(\od(G))$ maps injectively into $\out(\f/Z(\f))$, $\ker(\pi_\f \bar{\kappa}_G)$ is contained in $O_{2'}(\od(G))$, giving us that $\kapb_G$ is injective.

	\subsection{Preliminary results}
	
	We first show that $A/Z(\f)$ is a 2-centric subgroup of $G/Z(G)$. This follows from the following result of \cite{broto2019automorphisms} that states that $A$ is a 2-centric subgroup of $G$ and Theorem \ref{lem: generalized 27}. 
	\begin{lem}[Lemma 5.3 in \cite{broto2019automorphisms}]\label{lem: 53 in BMO19}
		Assuming Notation \ref{not: gp of lie type} and Hypotheses \ref{hyp: gp of lie type}, the following holds:
		\begin{enumerate}
			\item[(a)] $C_{W}(A)=1$, $C_{\G}(A)=C_{\G}(T)=\T$, $C_G(A)=T$, and $C_S(A)=A$. 
			\item[(b)] $\aut_G(A)=\aut_{W_0}(A)$. 
		\end{enumerate}
	\end{lem}
	\noindent We now prove that $A/Z(\f)$ is a 2-centric subgroup of $G/Z(G)$ by showing that it is a Sylow 2-subgroup of its centralizer. 
	\begin{lem}\label{lem: AZF is 2centric}
		Assume Hypotheses \ref{hyp: gp of lie type}, $A/Z(\f)$ is 2-centric.
	\end{lem}
	\begin{proof}
		Let $g\in G$ be such that $gZ(G)\in C_{G/Z(G)}(A/Z(\f))$. That is, $c_g|_{A/Z(\f)}=\id_{A/Z(\f)}$. Note that $c_g\in\aut_G(A)=\aut_{W_0}(A)$ by Lemma \ref{lem: 53 in BMO19} and is restricted to $A/Z(\f)$. By Theorem \ref{lem: generalized 27}, $W_0$ acts faithfully on $A/Z(\f)$ which yields that $c_g$ is trivial in $\aut_G(A)$. Therefore, $g\in C_G(A)=T$ by Lemma \ref{lem: 53 in BMO19}. Hence, $gZ(G)\in T/Z(G)$ and since $A/Z(\f)\in\syl_2(T/Z(G))$, $A/Z(\f)$ is 2-centric in $G/Z(G)$. 
	\end{proof}

	To discuss the image of $O_2(\od(G))$ in $\out(\f)$ where $O_2(\od(G))$ under $\kapb_G$, we make the following definition: 
	\begin{defn}[Follows Notation 5.2(H) of \cite{broto2019automorphisms}]
		We define 
		\begin{align*}
			\aut_{\text{diag}}(\f)=\{\b\in\aut(\f)\mid\b|_A=\id_A\}  
		\end{align*}
		and
		\begin{align*}
			\aut_{\text{diag}}(\f/Z(\f))=\{\b\in\aut(\f/Z(\f))\mid\b|_{A/Z(\f)}=\id_{A/Z(\f)}\}
		\end{align*}
		We then let $\out_{\text{diag}}(\f)$ (and subsequently $\out_{\text{diag}}(\f/Z(\f))$) to be the image of $\aut_{\text{diag}}(\f)$ in $\out(\f)$ (respectively, $\aut_{\text{diag}}(\f/Z(\f))$ in $\out(\f/Z(\f))$). 
	\end{defn}
	
	The fact that $\bar{\kappa}_G$ restricts to a well-defined homomorphism from $O_{2}(\od(G))$ to $\out_{\text{diag}}(\f)$ is shown in Lemma 5.9 of \cite{broto2019automorphisms}. We give an alternative proof of this fact. We first define $\A=O_2(\T)$ to be the 2-power torsion of the maximal torus $\T$ of the algebraic group $\G$. We now prove the following lemma:
	\begin{lem}\label{lem: there is a welldefined map from O2ODG to ODF}
		Assume Hypotheses \ref{hyp: gp of lie type}. The Sylow 2-subgroup $S$ of $G$ is normalized by $N_{\A}(G)$. In particular, each element of $O_2(\od(G))$ is represented by $c_a$ for some $a\in N_{\A}(G)$ such that $a$ is of 2-power order, normalizes $S$, and centralizes~$A$. 
	\end{lem}
	\begin{proof}
		Since $N_{\A}(G)$ embeds into $\od(G)$ which is a finite group, it follows that $N_{\A}(G)$ is a finite 2-group. Also, note that $ N_{\A}(G) $ normalizes both $ G $ and $ T $, so it normalizes $ N_G(T) $. Set $ \hat{N} = N_{\A}(G) N_G(T) $. \par 
		
		  It also follows that $ N_G(T) $ normalizes $ \A $ as 
		\begin{align*}
			N_G(T) \leq N_G(C_{\G}(T)) = N_G(\T) \leq N_{G}(\A)
		\end{align*}
		and, as a result, $N_{\A}(G)$ is normal in $\hat{N}$. We now show that $N_{\A}(G)$ normalizes $S$. In particular, $N_{\A}(S)\leq N_{\G}(S)$. \par		
		
		Since $ S $ is a 2-subgroup of $ \hat{N} $ ($S\leq N_G(T)$ by Hypotheses \ref{hyp: gp of lie type}), there is some Sylow 2-subgroup $ P $ of $ \hat{N} $ such that $ S\leq P $. Since $ N_{\A}(G) $ is a normal 2-subgroup of $ \hat{N} $, $ N_{\A}(G) \leq P $. Since $ N_{\A}(G) $ normalizes $ P $ and $ N_G(T) $, $ N_{\A}(G) $ normalizes $ P\cap N_G(T)=S $. \par 
		
		Let $[\a]\in O_2(\od(G))$ be such that $\a\in N_{\aut(G)}(S)$. Since the map $\innd(G)\to\od(G)$ is surjective, there is some $t\in\T$ such that $[\a]=[c_t|_G]$. In particular, there is some power of 2, say $2^n$, such that $(c_t)^{2^n}=c_{t^{2^n}}\in\inn(G)$; that is, $t^{2^n}\in G$. Take $|t^{2^n}|=2^m k$ where $k$ is odd and set $t_0=(t^{2^n})^{2^m}$. Set $a = tt_0^{-1}$ so that $a$ is of 2-power order, $[\a]=[c_t|_G]=[c_{a}|_G]$, and $a\in N_{\A}(G)$.  In particular, $a$ is an element of $N_{\A}(G)$ that normalizes $S$ (by above), and as $A\leq\A$ and $\A$ is commutative, $a$ centralizes $A$. 
	\end{proof}
	
	Towards the goal of proving that $O_2(\od(G))$ maps injectively to $\out_{\text{diag}}(\f/Z(\f))$ under $\pi_\f\bar{\kappa}_G$, we describe the structure of the subgroup $\out_{\text{diag}}(\f/Z(\f))$ in the following lemma:
	\begin{lem}\label{lem: outdiagF = autdiagF mod aut AZ SZ}
		Assuming Notation \ref{not: gp of lie type} and Hypotheses \ref{hyp: gp of lie type}, 
		\begin{align*}
			\out_{\text{diag}}(\f/Z(\f))\cong\aut_{\text{diag}}(\f/Z(\f))/\aut_{A/Z(\f)}(S/Z(\f)).
		\end{align*}
	\end{lem}
	\begin{proof}
		Denote the quotient by $Z(\f)$ or $Z(G)$ as follows: $\f^+=\f/Z(\f)$, $S^+=S/Z(\f)$, $A^+=A/Z(\f)$, and $G^+=G/Z(G)$. \par 
		
		First, set 
		\begin{align*}
			\aut_{\text{diag}}(S^+,A^+)=\{\a\in\aut(S^+)\mid\a=\id_{A^+},\text{ }[\a,S^+]\leq A^+\}
		\end{align*}
		and let $\out_{\text{diag}}(S^+,A^+)$ be its image in $\out(S^+)$. Since $A^+$ is centric in $S^+$ by Lemma \ref{lem: AZF is 2centric}, it follows that $\aut_{A^+}(S^+)=\inn(S^+)\cap\aut_{\text{diag}}(S^+, A^+)$. Therefore, $\out_{\text{diag}}(S^+,A^+)\cong\aut_{\text{diag}}(S^+,A^+)/\aut_{A^+}(S^+)$ which is naturally isomorphic to $H^1((S^+)/(A^+);A^+)$ (see 2.8.7 in \cite{suzuki1982group}) which is a $2$-group. Hence, $\aut_{\text{diag}}(S^+,A^+)$ and $\aut_{\text{diag}}(\f^+)$ are 2-groups. \par 
		
		Now, let $\beta$ be in $\aut_{\text{diag}}(\f^+)\cap \aut_{G^+}(S^+)$. That is, $\beta = c_{gZ(G)}$, $gZ(G)\in N_{G^+}(S^+)$, and $c_{gZ(G)}|_{A^+}=\id_{A^+}$. In particular, $gZ(G)$ centralizes $A^+$ and is of 2-power order as $\aut_{\text{diag}}(\f^+)$ is a 2-group. Therefore, $gZ(G)$ is an element of $S^+$, which gives us equality. In particular
		\begin{align*}
			\aut_{A^+}(S^+)& =\aut_{\text{diag}}(\f^+)\cap \inn(S^+) \\
			&=\aut_{\text{diag}}(\f^+)\cap \aut_{G^+}(S^+)
		\end{align*}
		Therefore, $\out_{\text{diag}}(\f^+)\cong \aut_{\text{diag}}(\f^+)/\aut_{A^+}(S^+)$.
	\end{proof}
	
	To show that $O_2(\od(G))$ injects into $\out_{\text{diag}}(\f/Z(\f))$, we need the following result from \cite{broto2019automorphisms}:

	\begin{lem}[Lemma 5.4 in \cite{broto2019automorphisms}]\label{lem: 54 in BMO}
		Assume Hypotheses \ref{hyp: gp of lie type}. 
		\begin{itemize}
			\item If all classes in $\hat{\Sigma}$ have order 1 or 2, then $C_{\T}(W_0)=C_{\T}(W)=Z(\G)$, and $Z(G)=C_T(W_0)$. 
			\item If $\hat{\Sigma}$ contains classes of order 3, then $\G\cong SL_{2n-1}(\overline{\F}_{q_0})$ and $G\cong SU_{2n-1}(q)$ for some $n\geq 2$. Also, $C_{\T}(W_0)\cong\overline{\F}_{q_0}^\times$, and $\sigma(t)=t^{-q}$ for all $t\in C_{\T}(W_0)$. 
		\end{itemize}
	\end{lem}
	
	Lemma \ref{lem: 54 in BMO} states that in all cases except for $G={}^2A_n^{\un}(q)$ with $n$ even, elements of $\T$ that centralize $W_0$ are elements of $Z(\G)$. We will use this fact to show that $O_2(\od(G))$ injects into $\out_{\text{diag}}(\f/Z(\f)))$. In the case of $G={}^2A_n^{\un}(q)$ with $n$ even, we use the following result:
	\begin{theo}[Theorem 2.5.12(c) in \cite{gorenstein2005classification}]\label{theo: outdiagG is isomorphic to Guniv}
		Let $G$ be an adjoint group of Lie type. Then $\od(G)$ is isomorphic to $Z(G^{\un})$.
	\end{theo}
	\noindent Note that $\od(G^{\ad})=\od(G^{\un})$. Therefore, in the case of $G={}^2A_n^{\un}(q)$ with $n$ even, $O_2(\od(G))\cong O_2(Z(G)) = Z(\f)$ (by Lemma \ref{lem: zf is 2 part of zg}). The center $Z({}^2A_n^{\un}(q))\cong Z_{(n+1,q+1)}$, and when $n$ is even, $O_2(Z({}^2A_n^{\un}(q)))=1$. Therefore, $O_2(\od({}^2A_n^{\un}(q)))=1$ injects into $\out(\f/Z(\f))$ trivially. 
	
	To show that $O_2(\od(G))$ injects into $\out_{\text{diag}}(\f/Z(\f))$ when $G$ is not ${}^2A_n^{\un}(q)$, we need the following result of Glauberman and Lynd from \cite{glauberman2016control}:
	\begin{lem}[Lemma 8.2 in \cite{glauberman2016control}]\label{lem: 82 in GL16}
		Let $N$ be a finite group with a Sylow $p$-subgroup $S$ and $A$ a normal $p$-subgroup of $N$ such that $C_N(A)\leq A$. Then $C_{\aut(N)}(S)=\aut_{Z(S)}(N)$. 
	\end{lem}
	\noindent We also need the following lemma:
	\begin{lem}\label{lem: setup for 82 from control}
		If $N$ is a finite group with a normal $p$-centric subgroup $A$, then $O_{p'}(N)=O_{p'}(C_N(A))$. Furthermore, $N/O_{p'}(N)$ is a group of characteristic $p$. 
	\end{lem}
	\begin{proof}
		First, since $A\norm N$, it follows that $C_N(A)\norm N$, and as a result $O_{p'}(C_N(A))\norm N$. Therefore, $O_{p'}(C_N(A))\leq O_{p'}(N)$. \par  
		
		Conversely, $A$ normalizes $O_{p'}(N)$ and $O_{p'}(N)$ normalizes $A$. Therefore, $[A,O_{p'}(N)]\leq A\cap O_{p'}(N)=1$. Ergo, $O_{p'}(N)\leq C_N(A)$ which gives us the desired equality. \par 
		
		Denote the image of a set $X$ by $X^{\dag}$ under the projection mod $O_{p'}(N)$. Now, note that we have the following string of equalities:
		\begin{align*}
			A^{\dag} = C_{N}(A)^{\dag} = C_{N^{\dag}}(A^{\dag}) 
		\end{align*}
		The first equality follows from the above argument. We show the second equality now. 
		
		First, note that it is clear that $C_{N}(A)^{\dag}\leq C_{N^{\dag}}(A^{\dag})$. Now, let $C$ be the preimage in $N$ of $C_{N^{\dag}}(A^{\dag})$. That is, $C$ is the largest subgroup of $N$ such that $[A,C]\leq O_{p'}(N)$. Furthermore, the preimage of $A^{\dag}$ in $N$ is $AO_{p'}(N)$ with $A$ being a Sylow $p$-subgroup of $AO_{p'}(N)$. Note that $AO_{p'}(N)\norm C$ as $A^{\dag}\norm C_{H^{\dag}}(A^{\dag})$. By a Frattini argument, $C=O_{p'}(N)A.N_C(A)=O_{p'}(N).N_C(A)$. This gives us that 
		\begin{align*}
			[A,C]\leq [A,O_{p'}(N)][A,N_C(A)]=[A,N_C(A)]\leq A\cap O_{p'}(N)=1,
		\end{align*}
		that is, $C\leq C_H(A)$. This gives the desired equality.
	\end{proof} 
	
	We can now prove that $O_2(\od(G))$ injects into $\out_{\text{diag}}(\f/Z(\f))$. 
	\begin{theo}\label{theo: generalized 59 from BMO19}
		Assume Hypotheses \ref{hyp: gp of lie type}. The map $\pi_\f\bar{\kappa}_G$ restricts to an injective homomorphism from $O_2(\od(G))$ to $\out_{\text{diag}}(\f/Z(\f))$. 
	\end{theo}
	\begin{proof}
		If $G={}^2A_n^{\un}(q)$, then $O_2(\od(G))=1$ as discussed above. Suppose that $G$ is not $G={}^2A_n^{\un}(q)$. \par 
		
		Let $[\a]\in\ker(\pi_\f\bar{\kappa}_G)\cap O_{2}(\od(G))$. By Lemma \ref{lem: there is a welldefined map from O2ODG to ODF}, there is some $a\in N_{\A}(G)$ such that we can take $\a=c_a|_G$. It is clear that $\pi_\f(\out_{\text{diag}}(\f))\leq\out_{\text{diag}}(\f/Z(\f))$. Since $\out_{\text{diag}}(\f/Z(\f))\cong\aut_{\text{diag}}(\f/Z(\f))/\aut_{A/Z(\f)}(S/Z(\f))$ and $[\a]\in\ker(\pi_\f\bar{\kappa}_G)$, there is some $b\in A$ such that $c_a|_{S/Z(\f)}=c_b|_{S/Z(\f)}$. Therefore, $c_{ab^{-1}}|_{S/Z(\f)}=\id_{S/Z(\f)}$ and $ab^{-1}\in C_{\G}(A)=\T$. In fact, $ab^{-1}\in\A$.  \par 
		
		Set $u=ab^{-1}$. Since $u$ normalizes both $T$ and $G$, $u\in N_{\A}(N_G(T))$. By Lemma \ref{lem: setup for 82 from control}, since $A/Z(\f)$ is a normal $2$-centric group of $N_G(T)/Z(G)$, we have the equality 
        \begin{align*}
            O_{2'}(N_G(T)/Z(G))=O_{2'}(C_{N_G(T)/Z(G)}(A/Z(\f)))=O_{2'}(T/Z(G)),
        \end{align*}
        and $N_G(T)/Z(G)O_{2'}(T)$ is a group of characteristic 2. We then can apply Lemma \ref{lem: 82 in GL16} to see that $c_u|_{N_G(T)/Z(G) O_{2'}(T)}=c_{vZ(\f)}$ for some $vZ(\f)\in Z(S/Z(\f))\leq A/Z(\f)$. \par 
		
		Adjusting, $[uv^{-1}, N_G(T)]\leq Z(G)O_{2'}(T)\leq T$. In particular, $uv^{-1}$ acts trivially on the Weyl group $W_0$. Therefore, $uv^{-1}\in C_{\T}(W_0)=Z(\G)$. This gives us that 
		\begin{align*}
			[\a] = [c_a|_G] = [c_u|_G] = [c_{uv^{-1}}|_G] = [\id_G]
		\end{align*}
		Ergo, $O_2(\od(G))$ maps injectively into $\out_{\text{diag}}(\f/Z(\f))$. 
	\end{proof}

	\subsection{The injectivity of \texorpdfstring{$\dot{\kappa}_G$}{kappa bullet}} 
	To prove that the map $\kapb_G$ is injective in most cases, we relate the kernel of $\pi_\f\bar{\kappa}_G$ to the kernel of another map that shows how the automorphisms of $G$ act on $A/Z(\f)$. Consider the following diagram:
	\[\begin{tikzcd}
		{\out(G)} && {\out(\f)} && {\frac{\aut(A,\f)}{\aut_{N_G(S)}(A)}} \\
		\\
		{\out(G/Z(G))} && {\out(\f/Z(\f))} && {\frac{\aut(A/Z(\f),\f)}{\aut_{N_G(S)}(A/Z(\f))}}
		\arrow["{\pi_G}", from=1-1, to=3-1]
		\arrow["{\bar\kappa_G}", from=1-1, to=1-3]
		\arrow["{\bar\kappa_{G/Z(G)}}", from=3-1, to=3-3]
		\arrow["{\pi_\f}", from=1-3, to=3-3]
		\arrow["\rho_1", from=1-3, to=1-5]
		\arrow["\rho_2", from=1-5, to=3-5]
	\end{tikzcd}\]
	\noindent where $\rho_1$ is given by restriction to $A$ and $\rho_2$ is given by restriction to $A/Z(\f)$. Set $\pi_{A/Z(\f)}=\rho_2\rho_1\bar{\kappa}_G$. The first key result is the following lemma: 
	\begin{lem}\label{lem: kerpiA/Z(f) contains kerpif kappaG}
		Assuming Hypotheses \ref{hyp: gp of lie type}, $\ker(\pi_{A/Z(\f)})\geq\ker(\pi_\f \bar{\kappa}_G)$. 
	\end{lem}
	\begin{proof}
		Let $\a\in N_{\aut(G)}(S)$ be such that $[\a]\in\ker(\pi_\f\bar\kappa_G)$. Therefore, there is some $gZ(G)\in N_{G/Z(G)}(S/Z(\f))$ such that $\a|_{S/Z(\f)}=c_{gZ(G)}|_{S/Z(\f)}$. That is, for every $t\in S$, there is some $u_t\in Z(\f)$ such that $g^{-1}tgu_t\in S$, and therefore, $g^{-1}tg\in S$ and $g\in N_G(S)$. Hence, $\a|_{A/Z(\f)}\in\aut_{N_G(S)}(A/Z(\f))$ which gives us that $\pi_{A/Z(\f)}([\a])=1$.
	\end{proof}
	
	Because of Lemma \ref{lem: kerpiA/Z(f) contains kerpif kappaG}, we calculate the kernel of $\pi_{A/Z(\f)}$ and show that when $G/Z(G)$ is not in one of the two families of exceptions, $\pi_{A/Z(\f)}$ has kernel $O_{2'}(\od(G))$. This will give us that the map $\kapb_G$ is injective. To calculate $\ker(\pi_{A/Z(\f)})$, we need the following proposition of Broto, M{\o}ller, and Oliver:
	\begin{prop}[Proposition 5.13 in \cite{broto2019automorphisms}]\label{prop: A is char in S}
		Assume Hypotheses \ref{hyp: gp of lie type}. The subgroup $A$ is characteristic in $S$, and is the unique abelian subgroup of $S$ of order $|A|$, except when $q\equiv5\pmod{8}$ and $G\cong Sp_{2n}(q)$ for some $n\geq 1$. In any case, each normal subgroup of $S$ isomorphic to $A$ is $N_G(S)$-conjugate to $A$.   
	\end{prop}
	\noindent This proposition states that every automorphism of $S$ can be restricted to an automorphism of $A$ up to an inner automorphism induced by an element of $N_G(S)$. Using this result, we can calculate the kernel of $\pi_{A/Z(\f)}$. 
	
	\begin{lem}\label{lem: kerpiA/Z(f) is outdiag}
		Assume \ref{hyp: gp of lie type}. The kernel of the map 
        \begin{align*}
            \pi_{A/Z(\f)}:\out(G)\to\aut(A/Z(\f),\f)/\aut_{N_S(G)}(A/Z(\f))
        \end{align*}
        is 
		\begin{align*}
			\od(G)\{[\psi]\mid\psi\in\Phi_G\Gamma_G\text{ and }\psi|_{A/Z(\f)}\in\aut_{W_0}(A/Z(\f))\}.
		\end{align*}
	\end{lem}
	\begin{proof}
		We first show that 
		\begin{align*}
			\ker(\pi_{A/Z(\f)})=\{[\vp]\mid\exists\tilde{\vp}\in N_{\aut(G)}(S)\text{ s.t. }\tilde{\vp}|_{A/Z(\f)}\in\aut_{W_0}(A/Z(\f))\text{ and }[\vp]=[\tilde{\vp}]\},
		\end{align*}
		where 
		\begin{align*}
			\aut_{W_0}(A/Z(\f))=\{\vp\in\aut(A/Z(\f))\mid\vp=\psi|_{A/Z(\f)}\text{ for some }\psi\in\aut_{W_0}(A)\}.
		\end{align*}
		First, let $[\vp]\in\out(G)$ be such that $[\vp]=[\tilde{\vp}]$ where $\tilde{\vp}\in N_{\aut(G)}(S)$ and $\tilde{\vp}|_{A/Z(\f)}\in\aut_{W_0}(A/Z(\f))$. Therefore, since $\aut_G(A)=\aut_{W_0}(A)$ by Lemma \ref{lem: 53 in BMO19}, there is some $g\in N_G(A)$ such that $\tilde{\vp}|_{A/Z(\f)}=c_g|_{A/Z(\f)}$. Since $S\leq N_G(A)$, there is some $h\in N_G(A)$ such that $gh\in N_G(S)$. Then $\tilde{\vp}c_h|_{A/Z(\f)}=c_{gh}|_{A/Z(\f)}$ and $[\vp]=[\tilde{\vp}]=[\tilde{\vp}c_h]\in\ker(\pi_{A/Z(\f)})$. The other inclusion is clear. \par
		
		We next show that
		\begin{align*}
			\ker(\pi_{A/Z(\f)})=\od(G)\{[\psi]\mid\psi\in\Phi_G\Gamma_G\text{ and }\psi|_{A/Z(\f)}\in\aut_{W_0}(A/Z(\f))\}. 
		\end{align*}
		Note that $ \out(G)=\od(G)\Phi_G\Gamma_G $ where $ \od(G)\cap\Phi_G\Gamma_G=1 $ by Lemma \ref{lem: 53 in BMO19}. Since $\od(G)$ is induced from $N_{\T}(G)$ and $A\leq\T$, $\od(G)\leq\ker(\pi_{A/Z(\f)})$. Let $[\psi]$ be such that $\psi\in\Phi_G\Gamma_G$ and $\psi|_{A/Z(\f)}\in\aut_{W_0}(A/Z(\f))$. Further, suppose that $\tilde{\psi}\in N_{\aut(G)}(S)$ such that $[\psi]=[\tilde{\psi}]$. Therefore, there is some $c_g\in\inn(G)$ such that $\tilde{\psi}=\psi c_g$, and since $\tilde{\psi}$ and $\psi$ act on $A/Z(\f)$, $c_g$ acts on $A/Z(\f)$. In particular, $c_g|_{A/Z(\f)}\in\aut_G(A/Z(\f))=\aut_{W_0}(A/Z(\f))$. Hence, $\tilde{\psi}|_{A/Z(\f)}=\psi c_g|_{A/Z(\f)}\in\aut_{W_0}(A/Z(\f))$. This shows that $\ker(\pi_{A/Z(\f)})\geq\od(G)\{[\psi]\mid\psi\in\Phi_G\Gamma_G\text{ and }\psi|_{A/Z(\f)}\in\aut_{W_0}(A/Z(\f))\}$. The other inclusion is clear. 
	\end{proof}
	
	We now show that $\kapb_G$ is injective for all the finite simple groups of Lie type with two families of exceptions. We will use this theorem to show that the simple groups of Lie type for which $\kapb_G$ is injective have centric linking systems with no noninner rigid automorphisms. 
	\begin{theo}\label{theo: extended 516 for most adjoint groups of lie type}
		Assume Hypotheses \ref{hyp: gp of lie type}. If $G/Z(G)$ is not one of $A_1^{\ad}(q)$ or ${}^2D_n^{\ad}(q)$ with $n\geq 3$, then the map $\kapb_G$ is injective.
	\end{theo}
	\begin{proof}
		By Theorem \ref{lem: generalized 27}, we know that $W_0$ acts on $A/Z(\f)$ faithfully; furthermore, since  $G$ is not $A_1^{\un}(q)$ or ${}^2D_n^{\un}(q)$, if there is $\g\in W_0\Gamma_G$ such that $\g|_{A/Z(\f)}=\psi_{q_0^r}|_{A/Z(\f)}$ for some $\psi_{q_0^r}\in\Phi_G$, then $q_0^r\equiv-1\pmod{\exp(A)}$. Since $q_0\equiv1\pmod{4}$ by Hypotheses \ref{hyp: gp of lie type}, $q_0^r\equiv1\pmod{\exp(A)}$. Therefore, $W_0\Phi_G\Gamma_G$ acts faithfully on $A/Z(\f)$. Thus, $\{[\psi]\mid\psi\in\Phi_G\Gamma_G\text{ and }\psi|_{A/Z(\f)}\in\aut_{W_0}(A/Z(\f))\}=1$ and $\ker(\pi_{A/Z(\f)})=\od(G)$. By Theorem \ref{theo: generalized 59 from BMO19}, $\pi_\f \bar{\kappa}_G$ maps $O_{2}(\od(G))$ injectively into $\out(\f)$. Hence, $\ker(\pi_\f \bar{\kappa}_G)\leq O_{2'}(\od(G))$.
	\end{proof}
	
	\subsection{Rigid automorphisms centric linking systems of the groups of Lie type}
	We now use the results from the previous section and results of \cite{broto2019automorphisms} to show that, up to equivalent 2-fusion, the centric linking systems of the groups of Lie type do not have noninner rigid automorphisms with the two families of exceptions: $A_1^{\ad}(q)$ and ${}^2D_n^{\ad}(q)$ with $n\geq 3$. We first give explicit constructions of the noninner rigid automorphisms of $A_1^{\ad}(q)=\PSL_2(q)$ and ${}^2D_n^{\ad}(q)=\POmega_{2n}^-(q)=\Omega_{2n}^-(q)$ with $n\geq 3$. 
	
	In both cases there is an automorphism of $G$ that, when restricted to $A/Z(\f)$, acts as the identity on $A/Z(\f)$. When $G=A_1^{\un}(q)$, the automorphism that induces a noninner rigid automorphism of the centric linking system of $G/Z(G)$ is a field automorphism. In the following theorem, we construct this field automorphism and show that it induces a rigid automorphism of the centric linking system associated to $\PSL_2(q)$. 
	\begin{theo}\label{theo: A1 has a rigid auto}
		Assume Hypotheses \ref{hyp: gp of lie type}. If $G/Z(G)$ is $A_1^{\ad}(q)$, then the centric linking system of $G/Z(G)$ has a noninner rigid automorphism induced by a field automorphism of $G$.
	\end{theo}
	\begin{proof} 
		Note that $G=SL_2(q)$ and 
		\begin{align*}
			T&=\left\langle\left[\renewcommand\arraystretch{.5} \begin{matrix}
				\lambda & 0 \\
				0 & \lambda^{-1}
			\end{matrix}\right]\right\rangle=\langle h_{\a_1}(\lambda)\rangle \hspace{.25in}
			A=\left\langle\left[\renewcommand\arraystretch{.5} \begin{matrix}
				\mu & 0 \\
				0 & \mu^{-1}
			\end{matrix}\right]\right\rangle=\langle h_{\a_1}(\mu)\rangle \\ 
			Z(G)&=\left\langle\left[\renewcommand\arraystretch{.5} \begin{matrix}
				-1 & 0 \\
				0 & -1
			\end{matrix}\right]\right\rangle=\langle h_{\a_1}(-1)\rangle
		\end{align*}
		where $\F_q^\times=\langle\lambda\rangle$ and $|\mu|=|\lambda|_2$. It then follows that the subgroup $S=A\langle w_0\rangle$ is a Sylow 2-subgroup of $G$ with 
		\begin{align*}
			w_0=\left[\renewcommand\arraystretch{.5} \begin{matrix}
				0 & -1 \\
				1 & 0
			\end{matrix}\right]
		\end{align*}
		Now, set $\psi=\vp_{q_0^{2^{l-1}}}$ and note that $\psi\in\Phi_G$. We now show that $[S,\psi|_S]\leq Z(G)$. \par
		
		It is clear that $\psi(w_0)=w_0$. For $h_{\a_1}(\mu)\in A$, we have the following: 
		\begin{align*}
			[\vp,h_{\a_1}(\mu)]&=\vp(h_{\a_1}(\mu))h_{\a_1}(\mu^{-1}) \\
			&=h_{\a_1}(\mu^{q_0^{2^{l-1}}})h_{\a_1}(\mu^{-1}) \\
			&=h_{\a_1}(\mu^{q_0^{2^{l-1}}-1}) \\
			&=h_{\a_1}(\mu^{2^{k-1}}) \\
			&=h_{\a_1}(-1)\in Z(\f)=Z(G)
		\end{align*}
		since $(q_0^{2^{l-1}}-1)_2=2^{l-1+2}=2^{k-1}$ by Lemma \ref{lem: 2-part of 5}. Hence, $[S,\psi]\leq Z(G)$ which gives that $[\psi]\in\ker(\pi_\f\bar{\kappa}_G)$. That $\pi_G$ is an isomorphism, $[\psi]$ is and element of order 2, $\ker(\kappa_{G/Z(G)})$ is a $2'$-group, and the diagram is commutative gives that $\pi_G(\kappa_{G/Z(\G)}([\psi]))$ is a nontrivial element of $\ker(\mu_{\cl/Z(\f)})$.
	\end{proof}
	
	For $G/Z(G)={}^2D_n^{\ad}(q)$, we construct an automorphism of the orthogonal group of the underlying vector space induced by the reflection in the hyperplane orthogonal to an anisotropic vector and show that this automorphism induces the noninner rigid automorphism for the centric linking system for ${}^2D_n^{\ad}(q)$. In order to construct the noninner rigid automorphism for the centric linking system for ${}^2D_n^{\ad}(q)$, we need to review some results related to orthogonal groups. 
	\begin{defn}[\cite{taylor1992geometry}]
		Let $V$ be a vector space over a field $k$ with a quadratic form $Q$ and associated bilinear form $\b$. An \emph{isotropic vector $u$} is such that $\b(u,u)=0$. An \emph{anisotropic vector $v$} is such that $\b(v,v)\neq0$. A subspace $V_0$ of $V$ is totally anisotropic if $V_0$ does not contain any nonzero isotropic vectors. 
	\end{defn} 
	\begin{defn}[\cite{taylor1992geometry}]
		The \emph{Witt index} of a quadratic form is the common dimension of a maximal totally isotropic subspace, i.e., a subspace that consists solely of isotropic vectors.
	\end{defn}
	\begin{defn}[\cite{taylor1992geometry}]
		Let $V$ be a vector space over a field $k$ with a quadratic from $Q$ and associated bilinear form $\b$. A \emph{hyperbolic pair} of vectors  is a pair of vectors $v,u$ such that $\b(v,u)=1$. A \emph{hyperbolic plane} is a space that is generated by a pair of hyperbolic vectors.
	\end{defn}
	\begin{theo}[Witt's Theorem]
		Let $V$ be an $m$-dimensional vector space over a field $k$ equipped with a nondegenerate quadratic, symmetric, alternating, or Hermitian form. Any isometry between two subspaces of $V$ extends to an isometry of $V$. 
	\end{theo}
	\begin{theo}[Theorem 11.4 in \cite{taylor1992geometry}]\label{theo: O2q is dihedral}
		For $\varepsilon=\pm1$, $O^{\varepsilon}(2,q)$ is a dihedral group of order $2(q-\varepsilon).$
	\end{theo}
	
	\begin{theo}\label{theo: twisted Dn has a rigid auto}
		Assume Hypotheses \ref{hyp: gp of lie type}. The centric linking system for $\POmega^{-}_{2n}(q)$ with $n\geq3$ has a noninner rigid automorphism induced by a reflection in an anisotropic vector.
	\end{theo}
	\begin{proof} 
		Let $V$ be a vector space over $\F_q$ of dimension $2n$ with quadratic form $Q$ with Witt index $m=n-1$ and associated bilinear form $\b$. We can decompose $V$ as follows:
		\begin{align*}
			V = V_1\perp V_2\perp\cdots\perp V_m\perp W
		\end{align*}
		where each $V_i$ is a hyperbolic plane and $W$ is totally anisotropic. Set $U=V_1\perp\cdots\perp V_m$. \par 
		
		Using Witt's Theorem and a counting argument (see Theorem 3 in \cite{carter1964sylow}), a Sylow 2-subgroup of $O(U)\times O(W)$ is a Sylow 2-subgroup of $O(V)$. Let $S=S_U\times S_W$ be such a Sylow 2-subgroup. By Theorem \ref{theo: O2q is dihedral}, $O(W)=O^-_2(q)\cong D_{2(q+1)}$. Therefore, since $q\equiv1\pmod{4}$, $S_W$ is a Klein 4 group.  \par 
		
		Fix an anisotropic vector $w\in W$. This gives us that $\langle-\id_W\rangle\times\langle r_w\rangle$ is a Sylow 2-subgroup of $O(W)$. Since $W$ is orthogonal to $U$, $c_{r_w}$ is the identity on $U$, and in particular, $S_U$. It is clear that $c_{r_w}$ restricts to the identity on $S_W$. Therefore, since $\Omega(V)$ is a characteristic subgroup of $O(V)$, $c_{r_w}|_{\Omega(V)}$ is an automorphism of $\Omega(V)$ that restricts to the identity on $S\cap\Omega(V)$. \par 
		
		We now check that $[c_{r_w}]\neq1$ in $\out(\Omega(V))$. If $[c_{r_w}]=1$, then there would be some $t\in\Omega(V)$ such that $r_wt^{-1}\in C_{O(V)}(\Omega(V))$. Note that $C_{O(V)}(\Omega(V))$ is normal in $O(V)$ and is of nilpotency class at most 2. Since $C_{O(V)}(\Omega(V))$ is normal in $O(V)$, either $C_{O(V)}(\Omega(V))$ is contained in $SO(V)$ or $C_{O(V)}(\Omega(V))=O(V)$. If $C_{O(V)}(\Omega(V))=O(V)$, then $O(V)$ is of nilpotency class at most 2, but this implies that $\Omega(V)\leq Z(O(V))$ which is not the case. Hence, $C_{O(V)}(\Omega(V))\leq SO(V)$. Therefore, $r_wt^{-1}\in SO(V)$, but $\det(r_wt^{-1})=-1$ so this is impossible. Therefore, $[c_{r_w}]\neq1$ in $\out(\Omega(V))$.  \par 
		
		Note that $O_{2'}(\POmega_{2n}^-(q))=1$, and as a result, $\ker(\kappa_G)$ is a $2'$-group by Theorem \ref{theo: ker kappaG is a p prime group}. Since $[c_{r_w}|_{\POmega_{2n}^-(q)}]$ is of order 2 and $\bar{\kappa}_G([c_{r_w}|_{\POmega_{2n}^-(q)}])=1$, we must have that $\kappa_{G}([c_{r_w}|_{\POmega_{2n}^-(q)}])\in\ker(\mu_{\cl})$. 
	\end{proof}
	
	We summarize the theorems \ref{theo: extended 516 for most adjoint groups of lie type}, \ref{theo: A1 has a rigid auto}, and \ref{theo: twisted Dn has a rigid auto} in the following corollary:
	\begin{cor}\label{cor: extension of 516 to adjoint gps of Lie type}
		Assume Hypotheses \ref{hyp: gp of lie type}. One and only one of the following holds:
		\begin{enumerate}
			\item The map $\kapb_G$ is injective. 
			\item $G = A_1^{\ad}(q)$ and $\cl$ has a noninner rigid automorphism. 
			\item $G = {}^2D_n^{\ad}(q)$ and $\cl$ has a noninner rigid automorphism. 
		\end{enumerate}
	\end{cor}
	
	In order to show that the simple groups of Lie type that satisfy Hypotheses \ref{hyp: gp of lie type} and are not in the families of exceptions have centric linking systems that do not have noninner rigid automorphisms, we review the theory of universal central extensions of finite groups.
	\begin{defn}[\cite{aschbacher2000finite}]
		A \emph{central extension} of a group $G$ is a pair $(H,\pi)$, where $H$ is a group and $\pi:H\to G$ is a surjective group homomorphism with $\ker(\pi)\leq Z(H)$. 
	\end{defn}
	\begin{defn}[\cite{aschbacher2000finite}]
		Let $G$ be a group with central extensions $(G_1,\pi_1)$ and $(G_2,\pi_2)$. A \emph{morphism of central extensions} $\a:(G_1,\pi_1)\to(G_2,\pi_2)$ is a group homomorphism $\a:G_1\to G_2$ such that $\pi_1=\pi_2\a$.  
	\end{defn}
	\begin{defn}[\cite{aschbacher2000finite}]
		A central extension $(\hat{G},\pi)$ of a group $G$ is \emph{universal} if for every central extension $(H,\sigma)$, there is a unique central extension morphism $\a:(\hat{G},\pi)\to(H,\s)$. 
	\end{defn}
	Note that a group has a universal central extension if and only if the group is perfect and in that case, the universal central extension is also perfect (see 33.2 in \cite{aschbacher2000finite}). First, we show that for any characteristic subgroup $Z$ of $Z(G)$, we have an isomorphism of groups $\out(G)\cong\out(G/Z)$. 
	\begin{lem}\label{lem: outG cong outZG}
		If $ G $ is the universal central extension of $ G/Z $ for a characteristic subgroup $Z\leq Z(G)$, then $ \pi_G:\out(G)\to\out(G/Z) $ is an isomorphism. 
	\end{lem}
	\begin{proof}
		Given any $ \b\in\aut(G/Z) $, we get the short exact sequence 
		\begin{align*}
			1\to Z\xrightarrow{i} G\xrightarrow{\b^{-1}\pi}G/Z\to 1
		\end{align*} 
		and since $ G $ is the universal central extension of $ G/Z $, there exists a unique group homomorphism $ \a:G\to G $ such that the following diagram commutes:
		\[\begin{tikzcd}
			1 && Z && G \\
			&&&&& \circlearrowleft & {G/Z} && 1 \\
			1 && Z && G
			\arrow[from=1-1, to=1-3]
			\arrow[from=1-3, to=1-5]
			\arrow["\pi", from=1-5, to=2-7]
			\arrow["\alpha"', from=1-5, to=3-5]
			\arrow[from=2-7, to=2-9]
			\arrow[from=3-1, to=3-3]
			\arrow[from=3-3, to=3-5]
			\arrow["{\beta^{-1}\pi}"', from=3-5, to=2-7]
		\end{tikzcd}\]
		We show that $\a$ is an automorphism of $G$. \par 
		Since $\pi$ is surjective, there is some $h\in G$ such that $\pi(h)=\b^{-1}\pi(g)$. By commutativity of the right hand triangle, $\b^{-1}\pi(g)=\pi(h)=\b^{-1}\pi\a(h)$. Therefore, $g=\a(h)z$ for some $z\in Z$. That is, $G=\a(G)Z$. Since $G$ is perfect, 
		\begin{align*}
			G=[G,G]=[\a(G)Z,\a(G)Z]=[\a(G),\a(G)]\leq\a(G)
		\end{align*}  
		This implies that $\a$ is surjective. \par 	
		Note that $\ker(\a)\leq Z$ as $\ker(\a)\leq\ker(\b^{-1}\pi\a)=\ker(\pi)=Z$, and since $G$ is a perfect group, $G=\ker(\a)K$ where $K$ is a perfect subgroup of $G$. Hence,
		\begin{align*}
			G=[G,G]=[\ker(\a)K,\ker(\a)K]=[K,K]=K
		\end{align*}	
		which implies that $\ker(\a)=1$. Ergo, $\a\in\aut(G)$ as desired. \par
		Since $Z\leq Z(G)$ is characteristic in $G$, $\a(Z)=Z$ and that $\a|_{G/Z}$ induces an automorphism on $G/Z$. The commutative right triangle gives us that $\a|_{G/Z}=\b$. Hence, $\aut(G)\cong\aut(G/Z)$ since $G$ is the universal central extension of $G/Z$. It is clear that under this map, $\inn(G)$ maps onto $\inn(G/Z)$ which gives us that $\inn(G)$ maps isomorphically onto $\inn(G/Z)$. Hence, $\out(G)\cong\out(G/Z)$. 
	\end{proof}
	Note that combining the above lemma (applied to $Z=Z(G)$) with Lemma \ref{lem: zf is 2 part of zg}, we get the following commutative diagram where $G$ is a universal group of Lie type and $G/Z(G)$ is an adjoint group of Lie type: 
	\begin{center}
		$\begin{tikzcd}
			{\out(G)} && {\out(\cl_S^c(G))} && {\out(\f_S(G))} \\
			\\
			{\out(G/Z(G))} && {\out(\cl/Z(\f))} && {\out(\f/Z(\f))}
			\arrow["{\kappa_G}",  from=1-1, to=1-3]
			\arrow["{\mu_\cl}",  from=1-3, to=1-5]
			\arrow["{\cong}"', from=1-1, to=3-1]
			\arrow["{\kappa_{G/Z(G)}}", from=3-1, to=3-3]
			\arrow["{\mu_{\cl/Z(\f)}}", from=3-3, to=3-5]
			\arrow["{\pi_\f}", from=1-5, to=3-5]
		\end{tikzcd}$
	\end{center}
	\noindent We will call this diagram ``Diagram A'' and we will state some properties of the various maps in Diagram A. \par 
	
	The first key feature of the diagram is that both $\mu_\cl$ and $\mu_{\cl/Z(\f)}$ are surjective maps due to Proposition \ref{prop: lim1 is there but lim2 gone}. The second key feature of Diagram A is that the maps $\kappa_G$ and  $\kappa_{G/Z(G)}$ are surjective as, under Hypotheses \ref{hyp: gp of lie type}, $G/Z(G)$ tamely realizes its fusion system if $G$ is not $G_2(q)$ with $q_0\neq3$. This is due to the following theorem of \cite{broto2019automorphisms}:
	\begin{theo}[Theorem B in \cite{broto2019automorphisms}]\label{theo: tameness and universals have no rigid autos}
		Fix a pair of distinct primes $p$ and $q_0$, and a finite group of Lie type over a field of characteristic $q_0$ of universal or adjoint type. Assume that the Sylow $p$-subgroups of $G$ are nonabelian. Then there is a prime $q_0^*\neq p$ and a finite group of Lie type $G^*$ over a field of characteristic $q_0^*$ of universal or adjoint type, respectively, as described in Theorem \ref{prop: justification} such that $G^*\sim_p G$ and $\kappa_{G^*}$ is split surjective. If, furthermore $p$ is odd or $G^*$ has universal type, then $\mu_{G^*}$ is an isomorphism and $\bar{\kappa}_{G^*}$ is also split surjective.
	\end{theo} 
	\noindent By Theorem \ref{theo: tameness and universals have no rigid autos}, the centric linking system of $G_2(q)$ has no noninner rigid automorphisms since it is of universal type and since it has trivial center, $G_2(q)/Z(G_2(q))=G_2(q)$. We also have that, under Hypotheses \ref{hyp: gp of lie type}, $\kappa_G$ and $\kappa_{G/Z(G)}$ are surjective. Therefore, via commutativity of Diagram A, the map $\pi_\f$ is also surjective. This ultimately gives us that the map $\kapb_G$ is an isomorphism (with the two families of exceptions), but we only need that $\kappa_{G/Z(G)}$ is surjective to prove that the centric linking system for $G/Z(G)$ has no noninner rigid automorphisms when $G/Z(G)$ is not $A_1^{\ad}(q)$ or ${}^2D_n^{\ad}(q)$ with $n\geq3$. \par 
	
	To prove that the centric linking system for most of the simple groups of Lie type do not have noninner rigid automorphisms, we combine Theorem \ref{theo: ker kappaG is a p prime group}, Lemma \ref{lem: outG cong outZG}, and Theorem \ref{theo: tameness and universals have no rigid autos}, we get the following result: 
	\begin{theo}\label{theo: rigid autos live in kermu}
		Assume Hypotheses \ref{hyp: gp of lie type}. If the centric linking system for $G/Z(G)$ has a noninner rigid automorphism $[\a]$, then there is an element $[\b]\in\out(G)$ of order 2 such that $\kappa_{G/Z(G)}(\pi_G([\b]))$. Furthermore, if $G/Z(G)$ is not one of $A_1^{\ad}(q)$ or ${}^2D_n^{\ad}(q)$ with $n\geq3$, then $\ker(\mu_{\cl/Z(\f)})=1$. 
	\end{theo}
	\begin{proof}
		Since $\kappa_{G/Z(G)}$ is surjective, there is an element $[\a']\in\out(G/Z(G))$ such that $\kappa_{G/Z(G)}([\a'])=[\a]$. We can choose $[\a']$ of order 2 since, by Theorem \ref{theo: ker kappaG is a p prime group}, $\ker(\kappa_G)$ is a $2'$-group and $\ker(\mu_{\cl/Z(\f)})$ is an elementary 2-group by Theorem \ref{theo: kermu is a 2 group}. Since $\pi_G$ is an isomorphism, there is an element of order 2 $[\b]\in\out(G)$ such that $\pi_G([\b])=[\a']$. \par 
		
		Note that $[\b]$, when restricted to $S/Z(\f)$, will be in the kernel of $\pi_\f\bar{\kappa}_G$. By Theorem \ref{theo: extended 516 for most adjoint groups of lie type}, if $G/Z(G)$ is not one of $A_1^{\ad}(q)$ or ${}^2D_n^{\ad}(q)$ with $n\geq3$, then $\ker(\pi_\f\bar{\kappa}_G)\leq O_{2'}(\od(G))$. Hence, $[\b]=1$ which gives us that $[\a]=1$. 
	\end{proof}
	
	We now want to show that the centric linking system associated to a quasisimple cover of a simple group of Lie type whose centric linking system does not have noninner rigid automorphism does not have a noninner rigid automorphism itself. To accomplish this, we need the following two results of Broto, M{\o}ller, Oliver, and Ruiz. We will use the following results to prove a general result that, if $G$ is a quasisimple group such that $G/Z(G)$ has a centric linking system with no noninner rigid automorphisms, then $G$ has a centric linking system with no noninner rigid automorphisms. We will use this result to complete the characterization of the groups of Lie type that satisfy $\ker(\mu_\cl)=1$. 
	\begin{prop}[Proposition 5.2 in \cite{broto2023realizabilitytamenessfusionsystems}]\label{prop: can switch out G/Z(G) for tameness}
		Fix a known simple group $G$, choose $S\in\syl_p(G)$, and assume that $S$ is not normal in $\f_S(G)$. Then $\f_S(G)$ is tamely realized by some known simple group $G^*$. 
	\end{prop}
	
	The following result of \cite{broto2023realizabilitytamenessfusionsystems} gives a result regarding tameness for $p'$-reduced $\mathcal{K}\mathcal{C}$-groups. A group $G$ is a $p'$-reduced if $O_{p'}(G)=1$ and is a $\mathcal{K}\mathcal{C}$-group if all of its components are known quasisimple groups. For our use, its important to note that simple groups are $p'$-reduced $\mathcal{K}\mathcal{C}$-groups and that quasisimple groups are $\mathcal{K}\mathcal{C}$-groups.  
	\begin{prop}[Proposition 5.3(a) in \cite{broto2023realizabilitytamenessfusionsystems}]\label{prop: can switch out G for tameness}
		Let $\f$ be a saturated fusion system over a finite $p$-group $S$. If $\f/Z(\f)$ is tamely realized by the finite $p'$-reduced $\mathcal{K}\mathcal{C}$-group $G^*$, then $\f$ is tamely realized by a finite $p'$-reduced $\mathcal{K}\mathcal{C}$-group $G$ such that $G^*\cong G/Z(G)$. 
	\end{prop}
	
	We can now prove the following lemma regarding quasisimple groups.
	\begin{lem}\label{lem: if a simple group has no noninner rigids, then any quasisimple cover has no noninner rigids}
		Let $G$ be a quasisimple group that is not 2-Goldschmidt and $S\in\syl_2(G)$. Set $\f=\f_S(G)$ and $\cl=\cl_S^c(G)$.  If the centric linking system for $G/Z(G)$ has no noninner rigid automorphisms, then $\cl$ has no noninner rigid automorphisms. 
	\end{lem}
	\begin{proof}
		Note by Lemma \ref{lem: zf is 2 part of zg}, $Z(\f)=O_{2}(Z(G))$ and, as a result, $\f/Z(\f)$ is the fusion system for $G/Z(G)$ over $S/Z(\f)$ and $\cl/Z(\f)$ is the centric linking system for $G/Z(G)$ over $S/Z(\f)$. \par 
		
		Using Proposition \ref{prop: can switch out G/Z(G) for tameness}, since $G/Z(G)$ is a simple group, we can replace it by another simple group $\overline{H}$ so that $\overline{H}$ tamely realizes $\f/Z(\f)$. Since $\overline{H}$ is a finite $2'$-reduced $\mathcal{K}\mathcal{C}$-group that tamely realizes $\f/Z(\f)$, we can apply Proposition \ref{prop: can switch out G for tameness} to replace $G$ by a $2'$-reduced $\mathcal{K}\mathcal{C}$-group $H$ so that $H$ tamely realizes $\f$ and $H/Z(H)\cong\overline{H}$. \par
		
		We then obtain the following commutative diagram:
		\begin{center}
			$\begin{tikzcd}
				{\out(H)} && {\out(\cl)} && {\out(\f)} \\
				\\
				{\out(\overline{H})} && {\out(\cl/Z(\f))} && {\out(\f/Z(\f))}
				\arrow["{\kappa_H}"', two heads, from=1-1, to=1-3]
				\arrow["{\pi_H}", from=1-1, to=3-1]
				\arrow["{\mu_{\cl}}"', two heads, from=1-3, to=1-5]
				\arrow["{\pi_\f}", from=1-5, to=3-5]
				\arrow["{\kappa_{\overline{H}}}", two heads, from=3-1, to=3-3]
				\arrow["{\mu_{\cl/Z(f)}}", two heads, from=3-3, to=3-5]
			\end{tikzcd}$
		\end{center}
		We want show that the map $\pi_H$ is injective, and in order to do so, we show that $H$ is perfect (i.e, quasisimple). We have the following: 
		\begin{enumerate}
			\item[(A)] $Z(H)$ is a $2$-group since $O_{2'}(H)=1$ ($H$ is $2'$-reduced). 
			\item[(B)] $\overline{H}$ is a nonabelian simple group since $G$ is not a 2-Goldschmidt group.
			\item[(C)] $\f$ is perfect since $G$ is perfect. By the focal subgroup theorem, $\mathfrak{foc}(\f)=[G,G]\cap S=S$.   
		\end{enumerate}
		By (C), $S\leq [H,H]$ which gives us that the normal closure of $S$ ($O^{2'}(H)$) is contained in $[H,H]$. Therefore, $H/[H,H]$ is a $2'$-group, which when coupled with (A), implies that $Z(H)\leq[H,H]$. Hence, 
		\begin{align*}
			[H,H]/Z(H)=[\overline{H},\overline{H}]=\overline{H}
		\end{align*}
		Therefore, $H$ is perfect as desired.\par
		
		We now show that $\pi_H$ is injective. Let $[\a]\in\ker(\pi_H)$. There is some $g\in H$ such that $\a c_{g^{-1}}|_{\overline{H}}=\id_{\overline{H}}$. In particular, $[H,\a c_{g^{-1}}]\leq Z(H)$. This gives us the following:
		\begin{align*}
			[H,H,\a c_{g^{-1}}]&=[H,\a c_{g^{-1}}] \\
			[H,\a c_{g^{-1}},H]&\leq [Z(H),H] = 1 \\
			[\a c_{g^{-1}},H,H]&\leq [Z(H),H] = 1
		\end{align*}
		and by the three subgroups lemma, $[H,\a c_{g^{-1}}]=1$ which gives us that $[\a]=1$. \par
		
		Now, let $[\a]\in\ker(\mu_\cl)$. Because $\kappa_H$ is surjective, we can choose an element $[\b]\in\out(H)$ of order 2 such that $\kappa_H([\b])=[\a]$. We then have that 
		\begin{align*}
			1&= \pi_\f(\mu_\cl(\kappa_H([\b]))) \\
			&=\mu_{\cl/Z(\f)}(\kappa_{\overline{H}}(\pi_H([\b])))
		\end{align*}
		Therefore, since $\ker(\mu_{\cl/Z(\f)})=1$, $[\b]\in\ker(\kappa_{\overline{H}}\pi_H)$. Since $\ker(\kappa_{\overline{H}})$ is a $2'$-group by Theorem \ref{theo: ker kappaG is a p prime group}, $[\b]\in\ker(\pi_H)=1$. Therefore, $[\a]=1$ and $\mu_{\cl}$ is injective.
	\end{proof}
	
	We can now prove the following corollary that classifies the finite groups of Lie type in odd characteristic that have centric linking systems that have noninner rigid automorphisms up to equivalent 2-fusion. 
	\begin{cor}\label{cor: classification of noninner rigids for gps of lie type}
		Let $G$ be a group of Lie type in odd characteristic. One and only one of the following holds:
		\begin{enumerate}
			\item The centric linking system of $G$ has no noninner rigid automorphisms. 
			\item $G$ has an equivalent fusion system to $A_1^{\ad}(q)$ with $q\equiv1\pmod{8}$. 
			\item $G$ has an equivalent fusion system to ${}^2D_n^{\ad}(q)$ with $q\equiv1\pmod{4}$. 
		\end{enumerate}
	\end{cor}
	\begin{proof}
		If $G$ has abelian Sylow $2$-subgroups, then its fusion system is constrained. Therefore, $\mu_{\cl_S^c(G)}$ is injective by Proposition \ref{prop: constrained fs have no noninner rigid autos}. \par 
		
		If $G$ is of universal type, we apply Theorem \ref{theo: tameness and universals have no rigid autos}. 
		
		\par If $G$ is of adjoint type and does not satisfy 2 or 3, we can find a universal group of Lie type $G^*$ such that $G^*$ satisfies Hypotheses \ref{hyp: gp of lie type} and $G^*/Z(G^*)\sim_2 G$ and then apply Theorem \ref{theo: rigid autos live in kermu}. 
		
		\par If $G$ satisfies 2 or 3, then the centric linking system of $G$ has a noninner rigid automorphism by Corollary \ref{cor: extension of 516 to adjoint gps of Lie type}. 
		
		\par If $G$ is a quasisimple extension of simple group of Lie type that is not universal, then it is a quasisimple extension of a simple group of Lie type whose centric linking system does not have a noninner rigid automorphism since $Z(A_1^{\un}(q))=Z_2$ and for $n\geq3$, $Z({}^2D_n^{\un}(q))=Z_2$. Therefore, the centric linking system of $G$ is has no noninner rigid automorphism by Lemma \ref{lem: if a simple group has no noninner rigids, then any quasisimple cover has no noninner rigids}.  
	\end{proof}

	%___________________________________________________________
	
	%------------------------------------------------------------------%
	% 6. Oliver's Conjecture for quasisimple saturated fusion systems  %
	%------------------------------------------------------------------%
	\section{Oliver's question for quasisimple saturated fusion systems}

	In Oliver's paper with Aschbacher (\cite{aschbacher2016fusion}), Oliver posed the following question:
	\begin{ques}[Oliver's Question 7.9 in \cite{aschbacher2016fusion}]
		Let $G$ be a finite 2-perfect group whose Schur multiplier has odd order. If $\f=\f_S(G)$ and $\cl=\cl_S^c(G)$ for some $S\in\syl_2(G)$, then is the homomorphism $\mu_\cl:\out(\cl)\to\out(\f)$ is injective?
	\end{ques}
    The following corollary shows that the answer to Oliver's question for the known quasisimple saturated fusion systems is ``yes''.
	\begin{repeatcor}{cor: quasisimple sat fus with noninner rigid autos}
		\textit{Let $\f$ be a known quasisimple saturated fusion system at the prime 2. The associated linking system of $\f$ has a noninner rigid automorphism if and only if there is a finite group $G$ and $S\in\syl_2(G)$ such that $\f\cong\f_S(G)$ and one of the following holds:
			\begin{enumerate}
				\item $G=A_1^{\ad}(q)=\PSL_2(q)$ with $q\equiv\pm1\pmod{8}$ 
				\item $G={}^2D_n^{\ad}(q)=\POmega^{-}_{2n}(q)$ with $n\geq3$ and $q$ odd
				\item $G=\alter(n)$ with $n\equiv2,3\pmod{4}$
		\end{enumerate}}
		\noindent \textit{Moreover, in each of these cases, $\ker(\mu_\cl)$ is cyclic of order 2.}
	\end{repeatcor} 
	\noindent Recall that the collection of the known quasisimple saturated fusion systems is defined to be the collection that contains the Benson-Solomon fusion systems and the fusion systems that are of the form $\f_S(G)$ where $G\in\mathfrak{Alt}\cup\mathfrak{Spor}^*\cup\mathfrak{Lie}^*$. The proof of Corollary \ref{cor: quasisimple sat fus with noninner rigid autos} is an application of Corollary \ref{cor: classification of noninner rigids for gps of lie type} in conjunction with the following results. \par 

	Levi and Oliver show that the morphism $\mu_\cl$ is injective for the Benson-Solomon fusion systems in the following lemma:
	\begin{lem}[Lemma 3.2 in \cite{levi2002construction}]\label{lem: bs systems taken care of}
		Fix a prime power $q$, and let 
		\begin{align*}
			&\mathcal{Z}_{\text{Sol}}(q):\mathcal{O}^c_{\text{Sol}}(q)\to\mathcal{A} & &\text{ and } & &Z_{\text{Spin}}(q):\mathcal{O}^c_{\text{Spin}}(q)\to\mathcal{A}
		\end{align*}
		be the functors $\mathcal{Z}(P)=Z(P)$ where $\mathcal{A}$ is the category of abelian groups. Then for all $i\geq0$, $\varprojlim_{\mathcal{O}^c_{\text{Sol}}(q)}^i(\mathcal{Z}_{\text{Sol}}(q))=0=\varprojlim_{\mathcal{O}^c_{\text{Spin}}(q)}^i(\mathcal{Z}_{\text{Spin}}(q))$.
	\end{lem}
    
	Anderson, Oliver, and Ventura show that $\mathfrak{Alt}$ satisfies Corollary \ref{cor: quasisimple sat fus with noninner rigid autos} in the following two lemmas:
	\begin{lem}[Lemma 4.6 in \cite{andersen2012reduced}]\label{lem: outf is Z2}
		Assume that $n\geq8$. Set $G=\alter(n)$ and fix $S\in\syl_2(G)$. Then
		\begin{align*}
			\out(\f)\cong\begin{cases}
				Z_2 &\text{if }n\equiv0,1\pmod{4} \\
				1   &\text{if }n\equiv2,3\pmod{4}
			\end{cases}
		\end{align*}
		In all cases, $\mu_\cl\circ\k_G$ sends $\out(G)=\out_{\sym(n)}(G)\cong Z_2$ onto $\out(\f)$. 
	\end{lem} 
	\begin{prop}[Proposition 4.8 in \cite{andersen2012reduced}]\label{prop: outg cong outcl}
		Set $G=\alter(n)$ and choose $S\in\syl_2(G)$. Then $\f_S(G)$ is tame. If $n\geq8$ and $n\equiv0,1\pmod{4}$, then 
		\begin{align*}
			\k_G:\out(G)\to\out(\cl_S^c(G))\cong Z_2
		\end{align*}
		is an isomorphism.
	\end{prop}
	\noindent In particular, Lemma \ref{lem: outf is Z2} and Proposition \ref{prop: outg cong outcl} together implies that $\mu_\cl$ is injective if and only if $n\equiv0,1\pmod{4}$.  \par 
	
	In \cite{broto2019automorphisms}, Oliver proves that for any group in $\mathfrak{Spor}^*$, the homomorphism $\mu_\cl$ is an isomorphism. 
	\begin{theo}[Theorem A in Oliver's section of \cite{broto2019automorphisms}]\label{theo: sporadic groups have kermu=1}
		Fix a sporadic simple group $G$, a prime $p$ which divides $|G|$, and $S\in\syl_p(G)$. Set $\f=\f_S(G)$ and $\cl=\cl_S^c(G)$. Then $\f$ is tame. Furthermore, $\kappa_G$ and $\mu_\cl$ are isomorphisms (hence $\f$ is tamely realized by $G$) if $p=2$ or if $p$ is odd and $S$ is nonabelian, with the following two exceptions:
		\begin{itemize}
			\item $G\cong M_{11}$ and $p=2$, in which case $\out(G)=1$ and $|\out(\f)|=|\out(\cl)|=2$; and 
			\item $G\cong He$ and $p=3$, in which case $|\out(G)|=2$ and $\out(\f)=\out(\cl)=1$
		\end{itemize}
	\end{theo}

	%___________________________________________________________

	%------------------%
	% ACKNOWLEDGEMENTS %
	%------------------%
	\section*{Acknowledgments}
	The author would like to thank Dr. Justin Lynd for his thoughtful supervision and guidance throughout the dissertation process. This paper, derived from that work, owes much to his expertise and support. The author would also like to thank an anonymous reviewer for their helpful notes. 
	%___________________________________________________________

	%------------%
	% REFERENCES %
	%------------%
	\printbibliography\clearpage

    \appendix
    \section{Tables and diagrams related to the finite groups of Lie type}
    The following tables and diagrams can be found in \cite{gorenstein2005classification}, but are helpful to reference when reading section 4. 
    \begin{table}[h!]
	\begin{align*}
		A_n: \Sigma&=\{\pm(a_i-a_j)\mid1\leq i<j\leq n+1\} \\
		\Pi&=\{a_1-a_2,a_2-a_3,\ldots,a_n-a_{n+1}\} \\
		B_n: \Sigma&=\{\pm a_i\pm a_j\mid1\leq i<j\leq n\}\cup\{\pm a_i\mid1\leq i\leq n\} \\
		\Pi&=\{a_1-a_2,a_2-a_3,\ldots,a_{n-1}-a_n,a_n\} \\
		C_n: \Sigma&=\{\pm a_i\pm a_j\mid1\leq i<j\leq n\}\cup\{\pm2 a_i\mid1\leq i\leq n\} \\
		\Pi&=\{a_1-a_2,a_2-a_3,\ldots,a_{n-1}-a_n,2a_n\} \\
		D_n: \Sigma&=\{\pm a_i\pm a_j\mid1\leq i<j\leq n\} \\
		\Pi&=\{a_1-a_2,a_2-a_3,\ldots,a_{n-1}-a_n,a_{n-1}+a_n\} \\
		G_2: \Sigma&=\{\pm\a_1,\pm\a_2,\pm(\a_1+\a_2),\pm(\a_1+2\a_2),\pm(\a_1+3\a_2),\pm(2\a_1+3\a_2)\} \\
		\Pi&=\{\a_1,\a_2\} \\
		F_4: \Sigma&=\{\pm a_i\pm a_j\mid1\leq i<j\leq4\}\cup\{\pm a_i\mid1\leq i\leq4\}\cup\{\a_\ep\mid\ep\in S^4\} \\
		\Pi&=\{a_2-a_3,a_3-a_4,a_4,a_{+---}\} \\
		E_8: \Sigma&=\{\pm a_i\pm a_j\mid 1\leq i<j\leq8\}\cup\left\{a_\ep\mid\ep\in S^8,\text{ }\prod_{i=1}^{8}\ep_i=1\right\} \\
		\Pi&=\{a_{+------+},a_7-a_8,a_6-a_7,a_7+a_8,a_5-a_6,a_4-a_5,a_3-a_4,a_2-a_3\} \\
		E_7: \Sigma&=\{\g\in\Sigma_{E_8}\mid\g\perp a_1+a_2\} \\
		\Pi&=\Pi_{E_8}-\{a_2-a_3\} \\
		E_6: \Sigma&=\{\g\in\Sigma_{E_7}\mid\g\perp a_2-a_3\} \\
		\Pi&=\Pi_{E_8}-\{a_3-a_4,a_2-a_3\} 
	\end{align*}
	\caption{The irreducible crystallographic root systems}
\end{table}
\begin{table}[h!]
	\centering
	\begin{tabular}{||c c||} 
		\hline
		System & Dynkin Diagram   \\ 
		\hline\hline
		$ A_1 $ & $ \dynkin[labels*={\a_1},scale=2.5] A{o} $ \\ 
		$ A_n, n\geq2 $ & $ \dynkin[labels*={\a_1,\a_2,\a_{n-1},\a_n},scale=2.5] A{oo.oo} $ \\
		$ B_n,n\geq3 $ & $ \dynkin[labels*={\a_1,\a_2,\a_{n-2},\a_{n-1},\a_n},scale=2.5] B{oo.ooo} $ \\
		$ C_n,n\geq2 $ & $ \dynkin[labels*={\a_1,\a_2,\a_{n-2},\a_{n-1},\a_n},scale=2.5] C{oo.ooo} $ \\
		$ D_n, n\geq5 $ & $ \dynkin[labels*={\a_1,\a_2, ,\a_{n-2},\a_{n-1},\a_n},scale=2.5] D{oo.oooo} $ \\ 
		$ D_4 $ & $ \dynkin[labels*={\a_1,\a_2,\a_{3},\a_{4}},scale=2.5] D{oooo} $ \\ 
		$ G_2 $ & $ \dynkin[labels*={\a_1,\a_2},scale=2.5] G{oo} $ \\ 
		$ F_4 $ & $ \dynkin[labels*={\a_1,\a_2,\a_{3},\a_{4}},scale=2.5] F{oooo} $ \\ 
		$ E_8 $ & $ \dynkin[labels*={\a_1,\a_4,\a_2,\a_3,\a_5,\a_6,\a_7,\a_8},scale=2.5] E{oooooooo} $ \\ 
		$ E_7 $ & $ \dynkin[labels*={\a_1,\a_4,\a_2,\a_3,\a_5,\a_6,\a_7},scale=2.5] E{ooooooo} $ \\
		$ E_6 $ & $ \dynkin[labels*={\a_1,\a_4,\a_2,\a_3,\a_5,\a_6},scale=2.5] E{oooooo} $ \\ 
		\hline
	\end{tabular}
	\caption{Dynkin diagrams for simple algebraic groups}
\end{table}
\clearpage
    \begin{table}[h!]
		\begin{center}
			\begin{tabular}{||c||c|c|c|c|c|c|c|c|c|c||}
				\hline 
				$ \Sigma $ & $ A_n $ & $ B_n $ & $ C_n $ & $ D_{2m} $ & $ D_{2m+1} $ & $ E_6 $ & $ E_7 $ & $ E_8 $ & $ F_4 $ & $ G_2 $ \\ 
				\hline
				$ Z(\G_u) $ & $ Z_{n+1} $ & $ Z_2 $ & $ Z_2 $ & $ Z_2\times Z_2 $ & $ Z_4 $ & $ Z_3 $ & $ Z_2 $ & $ 1 $ & $ 1 $ & $ 1 $ \\
				\hline
			\end{tabular}
		\end{center}
		\caption{The centers for the simple, simply connected algebraic groups}
	\end{table}
    \begin{table}[h!]
		\begin{center}
			\begin{tabular}{||c|l||}
				\hline
				$ \Sigma $ & Generators of $ Z(\G_u) $ \\ 
				\hline
				\hline
				$ A_n $ & $ z=h_{\a_1}(\w)h_{\a_2}(\w^2)\cdots h_{\a_n}(\w^n) $, $ \w $ a primitive $(n+1)^{\text{st}}$-root of 1 \\ 
				\hline
				$ B_n $ & $ z=h_{\a_n}(-1) $ \\ 
				\hline
				$ C_n $ & $ z=h_{\a_1}(-1)h_{\a_3}(-1)\cdots h_{\a_k}(-1) $, $ k=2\left\lfloor\frac{n-1}{2}\right\rfloor+1 $ \\ 
				\hline
				$ D_{2m} $ & $ z_1=h_{\a_1}(-1)h_{\a_3}(-1)\cdots h_{\a_{2m-1}}(-1) $ and $ z_2=h_{\a_{2m-1}}(-1)h_{\a_{2m}}(-1) $ \\ 
				\hline
				$ D_{2m+1} $ & $ z=h_{\a_1}(-1)h_{\a_3}(-1)\cdots h_{\a_{2m-1}}(-1)h_{\a_{2m}}(\w)h_{\a_{2m+1}}(-\w) $, $ \w^2=-1 $ \\ 
				\hline
				$ E_6 $ & $ z=h_{\a_1}(\w)h_{\a_2}(\w^2)h_{\a_5}(\w)h_{\a_2}(\w^2) $, $ \w^3=1 $ \\ 
				\hline
				$ E_7 $ & $ z=h_{\a_4}(-1)h_{\a_5}(-1)h_{\a_7}(-1) $ \\ 
				\hline
				Else & $Z(\G_u)=1$ \\
				\hline
			\end{tabular}
		\end{center}
		\caption{Generators for the centers of simple, simply connected algebraic groups}
	\end{table}
    \begin{table}[h!]
		\begin{center}
			\begin{tabular}{||c||c|c|c|c||}
				\hline 
				$ G $ & $ {}^2A^{\un}_n(q) $ &  $ {}^2D^{\un}_{n}(q) $ &  $ {}^2E_6^{\un}(q) $ & ${}^2B^{\un}_2(q)$, ${}^2F^{\un}_4(q)$, ${}^2G^{\un}_2(q)$  \\ 
				\hline
				$ Z(G) $ & $ Z_{(n+1,q+1)} $ & $ Z_{(4,q^n+1)} $ & $ Z_{(3,q+1)} $ & 1  \\
				\hline
			\end{tabular}
		\end{center}
		\caption{The centers for the Steinberg groups}
	\end{table} 
    
\end{document}